\documentclass[12pt,a4paper,twoside]{article}

\usepackage{a4wide, amssymb, amsmath, amsthm, graphics, comment, xspace, enumerate}
\usepackage{graphicx}
\usepackage[a4paper,colorlinks=true,citecolor=blue,urlcolor=blue,linkcolor=blue,bookmarksopen=true,unicode=true,
          pdffitwindow=true]{hyperref}
\usepackage{algorithmic, algorithm}

\input amssym.def
\input amssym

\allowdisplaybreaks

\newcommand{\NN}{\mathbb N}

\newcommand{\CC}{\mathbb C}
\newcommand{\RR}{\mathbb R}
\newcommand{\ZZ}{\mathbb Z}
\newcommand{\ds}{\displaystyle}
\newcommand{\EE}{\mathcal E}
\newcommand{\DD}{\mathcal D}

\newtheorem{theorem} {Theorem}[section]
\newtheorem{proposition}{Proposition}[section]
\newtheorem{lemma}{Lemma}[section]

\newtheorem{definition}{Definition}[section]
\newtheorem{remark}{Remark}[section]
\newcommand{\beq}{\begin{eqnarray}}
\newcommand{\eeq}{\end{eqnarray}}

\newcommand{\beqs}{\begin{eqnarray*}}
\newcommand{\eeqs}{\end{eqnarray*}}
\begin{document}

\title{Equivalence of several definitions of convolution \\ of  Roumieu ultradistributions}

\author{Stevan Pilipovi\'c \and Bojan Prangoski}
\date{}
\maketitle

\begin{abstract}
The equivalence of several definitions of convolution of two Roumieu ultradistributions is proved. For that purpose, the $\varepsilon$ tensor product of $\dot{\tilde{\mathcal{B}}}^{\{M_p\}}$ and a locally convex space $E$ is considered.
\end{abstract}

S. Pilipovi\'c, University of Novi Sad, Novi Sad, Serbia, stevan.pilipovic@gmail.com\\

B. Prangoski, University Ss. Cyril and Methodius, Skopje, Republic of Macedonia, bprangoski@yahoo.com

\section{Introduction}

Characterizations of the convolution of Roumieu type ultradistributions through the duality arguments for integrable ultradistributions was an open problem for many years. This paper gives such characterizations.

The existence of convolution of  distributions was considered by Schwartz \cite{SchwartzV}, \cite{SchwartzK} and later by many authors in various directions. In \cite{SchwartzV}, it is proved that if $S,T\in\DD'\left(\RR^d\right)$ are two distributions such that $S\otimes T\in\DD'_{L^1}\left(\RR^{2d}\right)$ then the convolutions $S*T$ can always be defined as an element of $\DD'\left(\RR^d\right)$. Later on, Shiraishi in \cite{Shiraishi} proved that this condition is equivalent to the condition that for every $\varphi\in\DD\left(\RR^d\right)$, $\left(\varphi*\check{S}\right)T\in\DD'_{L^1}\left(\RR^d\right)$. Many authors gave alternative definitions of convolution of two distributions and were proved that are equivalent to the Schwartz's definition (see, for example
\cite{DD}-\cite{K}, \cite{ort}-\cite{Wagner}, \cite{Shiraishi}, \cite{SI}). We refer also to an interesting resent paper  related to the existence of the convolution \cite{ort}. In the case of ultradistributions, the existence of convolution of two Beurling ultradistributions was studied  in \cite{PilipovicC} where the convolution is defined in analogous form to that of Schwartz. Later, in the case of Beurling ultradistributions,  in \cite{PK} was proved the equivalence of that definition and the analogous form of the Shiraishi's definition, as well as few other definitions. Our goal here is to give the equivalence of the Schwartz's and Shiraishi's definition in the case of Roumieu ultradistributions. Note that the sequential definitions or definitions through the Fourier transformations (and positions of wave fronts) are not considered in this paper although with the presented results such characterizations can be obtained. Our investigations are related essentially to the analysis of tensor product topologies within Roumieu type spaces so that the characterizations of convolution appear as  consequences of characterizations of  the $\varepsilon$ tensor product of $\dot{\tilde{\mathcal{B}}}^{\{M_p\}}$ and a locally convex space $E$.
Topological structure of Roumieu type spaces involves more complicated assertions and the proofs than in the case of Beurling type spaces where one can simply transfer the assertions from the Schwartz theory.
\section{Preliminaries}

The sets of natural (including zero), integer, positive integer, real and complex numbers are denoted by
$\NN$, $\ZZ$, $\ZZ_+$, $\RR$, $\CC$. We use the symbols for $x\in \RR^d$: $\langle x\rangle =(1+|x|^2)^{1/2} $,
$D^{\alpha}= D_1^{\alpha_1}\ldots D_n^{\alpha_d},\quad D_j^
{\alpha_j}={i^{-1}}\partial^{\alpha_j}/{\partial x}^{\alpha_j}\, ,$
$\alpha=(\alpha_1,\alpha_2,\ldots,\alpha_d)\in\NN^d$. A closed ball in $\RR^d$ with the center
at $x_0$ and radius $r>0$ is denoted by $K_{\RR^d}(x_0,r).$
As usual, in this theory, by  $M_{p}$, $p\in \NN$, $M_0=1$, is denoted  a  sequence  of  positive numbers for which we assume
(see \cite{Komatsu1}):
 $(M.1)$ $M_{p}^{2} \leq M_{p-1} M_{p+1}$, $p \in\ZZ_+$;
 $(M.2)$  $\ds M_{p} \leq c_0H^{p} \min_{0\leq q\leq p} \{ M_{p-q} M_{q}\}$, $p,q\in \NN$, for some $c_0,H\geq1$;
 $(M.3)$  $\ds\sum^{\infty}_{p=q+1}   M_{p-1}/M_{p}\leq c_0q M_{q}/M_{q+1}$, $q\in \ZZ_+$.\\
 For a multi-index $\alpha\in\NN^d$, $M_{\alpha}$ means $M_{|\alpha|}$, $|\alpha|=\alpha_1+...+\alpha_d$.

Let $U\subseteq\RR^d$ be an open set and $K\subset U$ be a compact set. We will always use the notation $K\subset\subset U.$
Recall,  $\EE^{\{M_p\},h}(K)$ is the space of all $\varphi\in \mathcal{C}^{\infty}(U)$ which satisfy $p_{K,h}(\phi)=\ds\sup_{\alpha\in\NN^d}\sup_{x\in K}\frac{|D^{\alpha}\varphi(x)|}{h^{\alpha}M_{\alpha}}<\infty$ and $\DD^{\{M_p\},h}_K$ as its subspace with elements supported by $K$.Then
$
\ds\EE^{\{M_p\}}(U)=\lim_{\substack{\longleftarrow\\ K\subset\subset U}}
\lim_{\substack{\longrightarrow\\ h\rightarrow \infty}} \EE^{\{M_p\},h}(K)$,
$\ds\DD^{\{M_p\}}_K=\lim_{\substack{\longrightarrow\\ h\rightarrow \infty}} \DD^{\{M_p\},h}_K$, $\ds\DD^{\{M_p\}}(U)=\lim_{\substack{\longrightarrow\\ K\subset\subset U}}\DD^{\{M_p\}}_K$.
The spaces of ultradistributions and ultradistributions with compact support of  Roumieu type are defined
as the strong duals of $\DD^{\{M_p\}}(U)$ and $\EE^{\{M_p\}}(U)$.
We refer to \cite{Komatsu1} for the properties of these spaces. In the sequel we will exclude  the notation $U$ when $U=\RR^d$.\\
\indent As in \cite{Komatsu1}, we define ultradifferential operators. It is said that $\ds P(\xi ) = \sum _{\alpha \in \NN^d} c_{\alpha } \xi ^{\alpha},\, \xi \in \RR^d$, is an
ultrapolynomial of the class  $\{M_{p}\}$, whenever the coefficients $c_{\alpha }$ satisfy
the estimate $|c_{\alpha }|  \leq C L^{\alpha }M_{\alpha}$, $\alpha \in \NN^d$,
for every $L > 0 $ and the corresponding $C_{L} > 0$. Then  $P(D)=\sum_{\alpha} c_{\alpha}D^{\alpha}$ is an
ultradifferential operator of the class $\{M_{p}\}$ and it acts continuously on
$\EE^{\{M_p\}}(U)$ and $\DD^{\{M_p\}}(U)$ and the corresponding spaces of ultradistributions.\\
\indent By $\mathfrak{R}$ is denoted a set of positive sequences which monotonically increases to infinity.
For $(t_j)\in\mathfrak{R}$, denote by $T_k$ the product $\ds \prod_{j=1}^k t_j$ and $T_0=1$. It is proved in \cite{Komatsu3} that the seminorms $\ds p_{K,(t_j)}(\varphi)=\ds \sup_{\alpha\in\NN^d}\sup_{x\in K} \frac{\left|D^{\alpha}\varphi(x)\right|}{T_{\alpha}M_{\alpha}}$, when $K$ ranges over the compact subsets of $U$ and $(t_j)$ in $\mathfrak{R}$, give the topology of $\EE^{\{M_p\}}(U)$. Also, for $K\subset\subset\RR^d$, the topology of $\DD^{\{M_p\}}_K$ is given by the seminorms
$p_{K,(t_j)}$, when $(t_j)$ ranges in $\mathfrak{R}$. From this it follows that $\ds \DD^{\{M_p\}}_K=\lim_{\substack{\longleftarrow\\ (t_j)\in\mathfrak{R}}} \DD^{M_p}_{K,(t_j)}$,
where $\DD^{M_p}_{K,(t_j)}$ is the Banach space of all $\mathcal{C}^{\infty}$ functions supported by $K$ for which the norm $p_{K,(t_j)}$ is finite. In the next sections we will need the following technical result.

\begin{lemma}\label{3}
Let $(t_p)\in\mathfrak{R}$. There exists $(t'_p)\in\mathfrak{R}$ such that $t'_p\leq t_p$ and $\ds\prod_{j=1}^{p+q}t'_j\leq 2^{p+q}\prod_{j=1}^{p}t'_j\cdot\prod_{j=1}^{q}t'_j$, for all $p,q\in\ZZ_+$.
\end{lemma}

\begin{proof} Define $t'_1=t_1$ and inductively $\ds t'_j=\min\left\{t_j,\frac{j}{j-1}t'_{j-1}\right\}$,
for $j\geq 2$, $j\in\NN$. The monotonicity of $(t_p)$ implies
$t'_{j+1}\geq t'_{j},\, j\in\ZZ_+$. To prove that $t'_j\rightarrow\infty$, assume the contrary. Then, for large enough $j_0$ and for $j>j_0,$ $
t'_j=(j/j_0)t'_{j_0}\rightarrow \infty,
$
$j\rightarrow \infty$, which is a contradiction;  so $(t'_j)\in\mathfrak{R}$. Note that, for all $p,j\in\ZZ_+$, we have
$
\ds t'_{p+j}\leq \frac{p+j}{j}t'_{j};
$ thus, $\ds T'_{p+q}=T'_p\cdot\prod_{j=1}^{q}t'_{p+j}\leq
\frac{(p+q)!}{p!q!}T'_pT'_{q}\leq 2^{p+q}T'_pT'_{q}$.\qed
\end{proof}
\indent As in \cite{PilipovicT}, we define $\tilde{\DD}^{\{M_p\}}_{L^{\infty}}\left(\RR^d\right)$ as the space of all $\mathcal{C}^{\infty}\left(\RR^d\right)$ functions such that, for every $(t_j)\in\mathfrak{R}$, the norm $\ds \|\varphi\|_{(t_j)}=\sup_{\alpha\in\NN^d}\sup_{x\in\RR^d}\frac{\left|D^{\alpha}\varphi(x)\right|}{T_{\alpha}M_{\alpha}}$ is finite. The space $\tilde{\DD}^{\{M_p\}}_{L^{\infty}}$ is complete Hausdorff locally convex space (from now on abbreviated as l.c.s.) because $\ds \tilde{\DD}^{\{M_p\}}_{L^{\infty}}=\lim_{\substack{\longleftarrow\\ (t_j)\in\mathfrak{R}}}\tilde{\DD}^{M_p}_{L^{\infty},(t_j)}$, where $\tilde{\DD}^{M_p}_{L^{\infty},(t_j)}$ is a Banach space ($(B)$ - space) of all $\mathcal{C}^{\infty}$ functions for which the norm $\|\cdot\|_{(t_j)}$ is finite. Denote by $\dot{\tilde{\mathcal{B}}}^{\{M_p\}}$ the completion of $\DD^{\{M_p\}}$ in $\tilde{\DD}^{\{M_p\}}_{L^{\infty}}$. The strong dual of $\dot{\tilde{\mathcal{B}}}^{\{M_p\}}$ will be denoted by $\tilde{\DD}'^{\{M_p\}}_{L^1}$. In the next sections we will need the following lemma that characterizes the elements of $\dot{\tilde{\mathcal{B}}}^{\{M_p\}}$.

\begin{lemma}\label{5}
$\varphi\in\dot{\tilde{\mathcal{B}}}^{\{M_p\}}$ if and only if $\varphi\in\tilde{\DD}^{\{M_p\}}_{L^{\infty}}$ and for every $\varepsilon>0$ and $(t_j)\in\mathfrak{R}$ there exists a compact set $K$ such that $\ds\sup_{\alpha\in\NN^d}\sup_{x\in\RR^d\backslash K} \frac{\left|D^{\alpha}\varphi(x)\right|}{T_{\alpha}M_{\alpha}}<\varepsilon$.
\end{lemma}

\begin{proof} Let $E$ be the subspace of $\tilde{\DD}^{\{M_p\}}_{L^{\infty}}$ defined by the conditions of the lemma.
 It is enough to prove that $E$ is complete and that $\DD^{\{M_p\}}$ is dense in $E$.
 By the standard arguments one can prove that $E$ is closed and so complete.
 The proof will be done if we prove that  $\DD^{\{M_p\}}$ is sequently dense in $E$.
  Let $\varphi\in E$. Take $\chi\in\DD^{\{M_p\}}$ such that $\chi=1$ on
the ball $K_{\RR^d}(0,1)$  and $\chi=0$ out of  $K_{\RR^d}(0,2)$. Then $\left|D^{\alpha}\chi(x)\right|\leq C_1h^{|\alpha|}M_{\alpha}$ for some $h>0$ and $C_1>0$. For $n\in\ZZ_+$, put $\chi_n(x)=\chi(x/n)$ and $\varphi_n=\chi_n\varphi$. Then $\varphi_n\in\DD^{\{M_p\}}$. Let $(t_j)\in\mathfrak{R}$. We have\\
$\ds \frac{\left|D^{\alpha}\varphi(x)-D^{\alpha}\varphi_n(x)\right|}{T_{\alpha}M_{\alpha}}$
\begin{eqnarray*}
&\leq& \frac{\left|1-\chi(x/n)\right|\left|D^{\alpha}\varphi(x)\right|}{T_{\alpha}M_{\alpha}}+
\sum_{\substack{\beta\leq\alpha\\ \beta\neq 0}} {\alpha\choose\beta}\frac{\left|D^{\beta}\chi(x/n)\right|\left|D^{\alpha-\beta}\varphi(x)\right|} {n^{|\beta|}T_{\alpha}M_{\alpha}}\\
&\leq&\frac{\left|1-\chi(x/n)\right|\left|D^{\alpha}\varphi(x)\right|}{T_{\alpha}M_{\alpha}}+
\frac{C_1\|\varphi\|_{(t_j/2)}}{n}\sum_{\substack{\beta\leq\alpha\\ \beta\neq 0}} {\alpha\choose\beta} \frac{h^{|\beta|}T_{\alpha-\beta}}{2^{|\alpha|-|\beta|}T_{\alpha}}\\
&\leq& \varepsilon+
\frac{C_1C_2\|\varphi\|_{(t_j/2)}}{n},\, n>n_0,
\end{eqnarray*}
independently of $x$ and $\alpha$, for large enough $n_0$. This implies the assertion.\qed
\end{proof}
Also we have the following easy fact.
\begin{lemma}\label{7}
The bilinear mapping  $\tilde{\DD}^{\{M_p\}}_{L^{\infty}} \times\dot{\tilde{\mathcal{B}}}^{\{M_p\}}\longrightarrow\dot{\tilde{\mathcal{B}}}^{\{M_p\}}$,
$(\varphi,\psi)\mapsto \varphi\psi$, is continuous.
\end{lemma}


\section{$\varepsilon$ tensor product of $\dot{\tilde{\mathcal{B}}}^{\{M_p\}}$ with a complete l.c.s.}

Let $E$ and $F$ be l.c.s. and
$\mathcal{L}_{c}(E,F)$  denote the space of continuous linear mappings from $E$ into $F$ with the topology of uniform convergence on convex circled compact
subsets of $E$. $E'_c$  denotes the dual of $E$ equipped with the topology of uniform convergence on convex circled compact subsets of $E$. As in  Komatsu \cite{Komatsu3} and Schwartz \cite{SchwartzV}, we define the $\varepsilon$ tensor product of $E$ and $F$, denoted by $E\varepsilon F$, as the space of all bilinear functionals on $E'_c \times F'_c$ which are hypocontinuous with respect to the equicontinuous subsets of $E'$ and $F'$. It is equipped with the topology of uniform convergence on products of equicontinuous subsets of $E'$ and $F'$. Moreover, the following isomorphisms hold:
\begin{equation}\label{1110}
E\varepsilon F \cong \mathcal{L}_{\epsilon}\left(E'_c, F\right)\cong \mathcal{L}_{\epsilon}\left(F'_c, E\right),
\end{equation}
where $\mathcal{L}_{\epsilon}\left(E'_c, F\right)$ is the space of all continuous linear mappings from $E'_c$ to $F$ equipped with the $\epsilon$ topology of uniform convergence on equicontinuous subsets of $E'$, similarly for $\mathcal{L}_{\epsilon}\left(F'_c, E\right)$. It is proved in \cite{SchwartzV}  that if both $E$ and $F$ are complete then $E\varepsilon F$ is complete. The tensor product $E\otimes F$ is injected in $E\varepsilon F$ under $(e\otimes f)(e',f')= \langle e,e'\rangle \langle f,f'\rangle$. The induced topology on $E\otimes F$ is the $\epsilon$ topology and we have the topological imbedding $E\otimes_{\epsilon} F\hookrightarrow E\varepsilon F$.\\
\indent  We recall the following definitions (c.f.  Komatsu \cite{Komatsu3} and Schwartz \cite{SchwartzV}).
The l.c.s. $E$ is said to have the sequential approximation property (resp. the weak sequential approximation property) if the identity mapping $\mathrm{Id}: E\longrightarrow E$ is in the sequential limit set (resp. the sequential closure) of $E'\otimes E$ in $\mathcal{L}_c(E,E)$.
The l.c.s. $E$ is said to have the weak approximation property if the identity mapping $\mathrm{Id}: E\longrightarrow E$ is in the closure of $E'\otimes E$ in $\mathcal{L}_c(E,E)$.
We also need the next proposition (\cite{Komatsu3}, proposition 1.4., p. 659).
\begin{proposition}
If $E$ and $F$ are complete l.c.s. and if either $E$ or $F$ has the weak approximation property then $E\varepsilon F$ is isomorphic to $E\hat{\otimes}_{\epsilon} F$.
\end{proposition}

For $K\subset\subset \RR^d$, we denote by $\mathcal{C}_0(K)$ the $(B)$ - space of all continuous functions supported by $K$ endowed with
 $\|\cdot\|_{L^{\infty}}$ norm.

\begin{lemma}\label{110}
Let $K_1$ and $K_2$ be two compact subsets of $\RR^d$ such that $K_1 \subset\subset \mathrm{int} K_2$. Then there exists a sequence $S_n$ of $\left(\mathcal{C}_0(K_1)\right)'\otimes \mathcal{C}_0(K_2)$ such that $S_n\longrightarrow \mathrm{Id}$, when $n\longrightarrow \infty$, in $\mathcal{L}_c\left(\mathcal{C}_0(K_1),\mathcal{C}_0(K_2)\right)$.
\end{lemma}

\begin{proof} For every $n\in \ZZ_+$, choose a finite open covering $\left\{U_{1,n},...,U_{k_n,n}\right\}$ of $K_1$ of open sets
each with diameter less than $1/n$ such that $\bar{U}_{j,n}\subseteq \mathrm{int}K_2$, $j=1,..., k_n$. Let $\chi_{j,n}$, $j=1,...,k_n$, be a
continuous partition of unity subordinated to $\left\{U_{1,n},...,U_{k_n,n}\right\}$. For every $j\in\{1,...,k_n\}$, choose a point
$x_{j,n}\in \mathrm{supp\,} \chi_{j,n}\cap K_1$. Define
$\ds S_n=\sum_{j=1}^{k_n}\delta\left(\cdot-x_{j,n}\right)\otimes
\chi_{j,n}\in \left(\mathcal{C}_0(K_1)\right)'\otimes \mathcal{C}_0(K_2).$
Let $V=\left\{\varphi\in \mathcal{C}_0(K_2)|\|\varphi\|_{L^{\infty}}\leq \varepsilon\right\}$ and $B$ a compact convex circled subset of $\mathcal{C}_0(K_1)$.
 Let $\mathcal{M}(B,V)=\left\{T\in \mathcal{L}\left(\mathcal{C}_0(K_1), \mathcal{C}_0(K_2)\right)|T(B)\subseteq V\right\}$.
  By the Arzela - Ascoli theorem, for the chosen $\varepsilon$ there exists $\eta>0$ such that for all $x,y\in K_1$ such that $|x-y|<\eta$, $|\varphi(x)-\varphi(y)|\leq \varepsilon$ for all $\varphi\in B$. Let $n_0\in\NN$ is so large such that $1/n_0< \eta$. Then, for $n\geq n_0$ and $x\in K_1$, we have
\begin{eqnarray*}
\left|S_n(\varphi)(x)-\varphi(x)\right|&=&\left|\sum_{j=1}^{k_n}\varphi\left(x_{j,n}\right)\chi_{j,n}(x)- \sum_{j=1}^{k_n}\varphi(x)\chi_{j,n}(x)\right|\\
&\leq& \sum_{j=1}^{k_n}\left|\varphi\left(x_{j,n}\right)-\varphi(x)\right|\chi_{j,n}(x)\leq \varepsilon,
\end{eqnarray*}
for all $\varphi\in B$. Note that, for $x\in K_2\backslash K_1$, $\varphi(x)=0$ and $\ds \left|S_n(\varphi)(x)\right|\leq \sum_{j=1}^{k_n}\left|\varphi\left(x_{j,n}\right)\right|\chi_{j,n}(x)\leq \varepsilon$. So, $S_n-\mathrm{Id}\in\mathcal{M}(B,V)$ for $n\geq n_0$.\qed
\end{proof}

\begin{lemma}\label{120}
$B$ is a precompact subset of $\dot{\tilde{\mathcal{B}}}^{\{M_p\}}\left(\RR^d\right)$ if and only if $B$ is bounded in $\dot{\tilde{\mathcal{B}}}^{\{M_p\}}\left(\RR^d\right)$ and for every $\varepsilon>0$ and $(t_j)\in\mathfrak{R}$, there exists $K\subset\subset\RR^d$ such that
$$
\ds\sup_{\varphi\in B}\sup_{\substack{\alpha\in \NN^d\\ x\in\RR^d\backslash K}}\frac{\left|D^{\alpha} \varphi(x)\right|}{T_{\alpha} M_{\alpha}}\leq \varepsilon.
$$
\end{lemma}

\begin{proof} $\Rightarrow$. Let $\varepsilon>0$ and $(t_j)\in\mathfrak{R}$ and $V=\left\{\varphi\in\dot{\tilde{\mathcal{B}}}^{\{M_p\}}|\|\varphi\|_{(t_j)}\leq\varepsilon/2\right\}$.
There
exist $\varphi_1,...,\varphi_n\in B$ such that for each $\varphi\in B$ there exists $j\in\{1,...,n\}$ such that $\varphi\in \varphi_j+V$.
Let $K\subset\subset\RR^d$ such that $\ds\left|D^{\alpha}\varphi_j(x)\right|/(T_{\alpha}M_{\alpha})\leq \frac{\varepsilon}{2}$
for all $x\in \RR^d\backslash K$, $\alpha\in\NN^d$, $j\in\{1,...,n\}$. Let $\varphi\in B$. There exists $j\in\{1,...,n\}$ such that
$\|\varphi-\varphi_j\|_{(t_j)}\leq\varepsilon/2$. The proof follows from
\begin{eqnarray*}
\frac{\left|D^{\alpha}\varphi(x)\right|}{T_{\alpha}M_{\alpha}}\leq \frac{\left|D^{\alpha}
\left(\varphi(x)-\varphi_j(x)\right)\right|}{T_{\alpha}M_{\alpha}}+ \frac{\left|D^{\alpha}\varphi_j(x)\right|}{T_{\alpha}M_{\alpha}}
\leq \frac{\varepsilon}{2}+\frac{\varepsilon}{2}=\varepsilon,\, x\in\RR^d\backslash K, \alpha \in\NN^d.
\end{eqnarray*}

$\Leftarrow$.
Let $V=\left\{\varphi\in \dot{\tilde{\mathcal{B}}}^{\{M_p\}}\left(\RR^d\right)|\|\varphi\|_{(t_j)}
\leq \varepsilon\right\}$. Since $B$ is bounded  in the Montel space $\EE^{\{M_p\}}\left(\RR^d\right)$, it is
precompact in $\EE^{\{M_p\}}\left(\RR^d\right)$.
Thus, there exists a finite subset $B_0=\{\varphi_1,...,\varphi_n\}$ of $B$ such that, for every $\varphi\in B$, there exists $j\in\{1,...,n\}$ such that $p_{K,(t_j)}(\varphi-\varphi_j)\leq \varepsilon$. If $\varphi\in B$ is fixed, take such $\varphi_j\in B_0$. Then, $\ds \frac{\left|D^{\alpha}\varphi(x)-D^{\alpha}
 \varphi_j(x)\right|}{T_{\alpha}M_{\alpha}}\leq \varepsilon$, for all $x\in K, \alpha\in\NN^d$. Also, by the assumption,
\begin{eqnarray*}
\frac{\left|D^{\alpha}\varphi(x)-D^{\alpha}\varphi_j(x)\right|}{T_{\alpha}M_{\alpha}}\leq \frac{\left|D^{\alpha}\varphi(x)\right|}{T_{\alpha}M_{\alpha}}+ \frac{\left|D^{\alpha}\varphi_j(x)\right|}{T_{\alpha}M_{\alpha}}\leq \frac{\varepsilon}{2}+\frac{\varepsilon}{2}=\varepsilon,
\end{eqnarray*}
for all $x\in\RR^d\backslash K,\, \alpha\in \NN^d$. So, the proof follows.\qed
\end{proof}

\begin{proposition}\label{pr11}
$\dot{\tilde{\mathcal{B}}}^{\{M_p\}}\left(\RR^d\right)$ has the weak sequential approximation property.
\end{proposition}

\begin{proof} Let $K_n=K_{\RR^d}(0,2^{n-1})$, $n\geq 1$. Let $\theta\in \DD^{\{M_p\}}_{K_1}$
is such that $\theta=1$ on $K_{\RR^d}(0,1/2)$. Define $\theta_n(x)=\theta(x/2^n)$, $n\in\ZZ_+$. Then $\theta_n\in\DD^{\{M_p\}}_{K_{n+1}}$
and $\theta_n=1$ on $K_n$. Let $T_n\in \mathcal{L}\left(\dot{\tilde{\mathcal{B}}}^{\{M_p\}}\left(\RR^d\right),
\dot{\tilde{\mathcal{B}}}^{\{M_p\}}\left(\RR^d\right)\right)$, defined by $T_n(\varphi)=\theta_n\varphi$. Let $\mu\in \DD^{\{M_p\}}_{K_1}$, $\mu\geq 0$,
is such that $\ds\int_{\RR^d}\mu(x)dx=1$ and define a delta sequence $\mu_m=m^d\mu(m\cdot)$, $m\in\ZZ_+$.  For each fixed $n\in \ZZ_+$,
by lemma \ref{110}, we  find
$$S_{k,n}=\sum_{l=1}^{j_{k,n}}\delta\left(\cdot-x_{l,k,n}\right)\otimes \chi_{l,k,n}\in
\left(\mathcal{C}_0(K_{n+1})\right)'\otimes \mathcal{C}_0(K_{n+2})$$ such that $S_{k,n}\longrightarrow \mathrm{Id}$,
when $k\longrightarrow \infty$, in $\mathcal{L}_c\left(\mathcal{C}_0(K_{n+1}),\mathcal{C}_0(K_{n+2})\right)$, where $\chi_{l,k,n}$
are continuous function with values in $[0,1]$ that have compact support in $\mathrm{int} K_{n+2}$ and $x_{l,k,n}$ are
points in $\mathrm{supp\,} \chi_{l,k,n}\cap K_{n+1}$. Moreover the support of $\chi_{l,k,n}$ has diameter less then $1/k$
and $\ds \sum_{l=1}^{j_{k,n}}\chi_{l,k,n}(x)\leq 1$ on $K_{n+2}$ and $\ds \sum_{l=1}^{j_{k,n}}\chi_{l,k,n}(x)=1$ on $K_{n+1}$.
 Define, $k,m,n \in\ZZ_+$,
$$
T_{k,m,n}= \sum_{l=1}^{j_{k,n}}\theta_n\delta\left(\cdot-x_{l,k,n}\right)\otimes(\mu_m*\chi_{l,k,n})
\mbox{ and }
T_{m,n}: \varphi\mapsto T_{m,n}(\phi) = \mu_m*(\theta_n\varphi).$$
First we prove that for each fixed $m,n\in\ZZ_+$, $T_{k,m,n}\rightarrow T_{m,n}$,
when $k\rightarrow \infty$, in\newline
$\mathcal{L}_c\left(\dot{\tilde{\mathcal{B}}}^{\{M_p\}}\left(\RR^d\right), \dot{\tilde{\mathcal{B}}}^{\{M_p\}}\left(\RR^d\right)\right)$.
Let $V=\left\{\varphi\in\dot{\tilde{\mathcal{B}}}^{\{M_p\}}\left(\RR^d\right) \big|\|\varphi\|_{(t_j)}\leq\varepsilon\right\}$, $B$ a convex circled compact subset of $\dot{\tilde{\mathcal{B}}}^{\{M_p\}}\left(\RR^d\right)$
 and
$$
\mathcal{M}(B,V)=\left\{T\in\mathcal{L}\left(\dot{\tilde{\mathcal{B}}}^{\{M_p\}}\left(\RR^d\right), \dot{\tilde{\mathcal{B}}}^{\{M_p\}}\left(\RR^d\right)\right)\Big|T(B)\subseteq V\right\}
$$
(a neighborhood of zero in $\mathcal{L}_c\left(\dot{\tilde{\mathcal{B}}}^{\{M_p\}}\left(\RR^d\right), \dot{\tilde{\mathcal{B}}}^{\{M_p\}}\left(\RR^d\right)\right)$).
Let $\varphi\in B$. Then\\
$\ds \frac{\left|D^{\alpha}T_{k,m,n}(\varphi)(x)-D^{\alpha}T_{m,n}(\varphi)(x)\right|}{T_{\alpha}M_{\alpha}}$ ($\alpha\in\NN^d,x\in\RR^d$)
\begin{eqnarray*}
&=&\frac{1}{T_{\alpha}M_{\alpha}}\left|\left(D^{\alpha}\mu_m\right)* \left(\sum_{l=1}^{j_{k,n}}\theta_n\left(x_{l,k,n}\right) \varphi\left(x_{l,k,n}\right)\chi_{l,k,n}-\theta_n\varphi\right)(x)\right|\\
&\leq& m^d\|\mu\|_{(t_j/m)}\int_{K_{n+2}}\sum_{l=1}^{j_{k,n}} \left|\theta_n\left(x_{l,k,n}\right)\varphi\left(x_{l,k,n}\right)-\theta_n(y)\varphi(y)\right|\chi_{l,k,n}(y)dy.
\end{eqnarray*}
Because the mapping $\varphi\mapsto \theta_n\varphi$, $\dot{\tilde{\mathcal{B}}}^{\{M_p\}}\left(\RR^d\right)\longrightarrow \mathcal{C}_0(K_{n+1})$ is continuous, it maps the compact set $B$ in a compact set in $\mathcal{C}_0(K_{n+1})$, which we denote by $B_1$. By the Arzela - Ascoli theorem, for the chosen $\varepsilon$ there exists $\eta>0$ such that for all $x,y\in K_{n+1}$ such that
$$|x-y|<\eta \Rightarrow \ds|\theta_n(x)\varphi(x)-\theta_n(y)\varphi(y)|\leq \frac{\varepsilon}{m^d\|\mu\|_{(t_j/m)}|K_{n+2}|}, \varphi\in B.$$ If we take $k_0$ large enough such that $1/k_0<\eta$, then, for all $k\geq k_0$,
\begin{eqnarray*}
\frac{\left|D^{\alpha}T_{k,m,n}(\varphi)(x)-D^{\alpha}T_{m,n}(\varphi)(x)\right|}{T_{\alpha}M_{\alpha}}\leq \varepsilon, \, x\in \RR^d, \alpha\in\NN^d, \varphi\in B.
\end{eqnarray*}
That is $T_{k,m,n}- T_{m,n}\in \mathcal{M}(B,V)$ for all $k\geq k_0$.
Now we  prove that, for each fixed $n\in\ZZ_+$, $T_{m,n}\longrightarrow T_n$, when $m\longrightarrow \infty$, in $\mathcal{L}_c\left(\dot{\tilde{\mathcal{B}}}^{\{M_p\}}\left(\RR^d\right), \dot{\tilde{\mathcal{B}}}^{\{M_p\}}\left(\RR^d\right)\right)$. We use the notation as above. Because of lemma \ref{3}, without losing generality,
we can assume that $(t_j)$ is such that $T_{p+q}\leq 2^{p+q}T_p T_q$, for all $p,q\in\NN$. Then, for $\varphi\in B, \alpha\in\NN^d,x\in\RR^d,$
\begin{eqnarray*}
\frac{\left|D^{\alpha}T_{m,n}(\varphi)(x)-D^{\alpha}T_n(\varphi)(x)\right|}{T_{\alpha}M_{\alpha}}\leq \int_{\RR^d} \mu_m(y)\frac{\left|D^{\alpha}(\theta_n\varphi)(x-y)-D^{\alpha}(\theta_n\varphi)(x)\right|}{T_{\alpha}M_{\alpha}}dy.
\end{eqnarray*}
Let $t'_1=t_1/(4H)$ and $t'_p=t_{p-1}/(2H)$, for $p\in\NN$, $p\geq 2$. Then $(t'_j)\in\mathfrak{R}$. For the moment, denote $\theta_n\varphi$ by $\varphi_n$. By the mean value theorem, we have
\begin{eqnarray*}
\left|D^{\alpha}\varphi_n(x-y)-D^{\alpha}\varphi_n(x)\right|&\leq&\sqrt{d}\|\varphi_n\|_{(t'_j)}T'_{|\alpha|+1} M_{|\alpha|+1}||y|\\
&\leq& \frac{c_0t_1M_1\sqrt{d}\|\varphi_n\|_{(t'_j)}T_{\alpha}M_{\alpha}|y|}{2}.
\end{eqnarray*}
Note that $\|\varphi_n\|_{(t'_j)}\leq \|\theta\|_{(t'_j/2)}\|\varphi\|_{(t'_j/2)}$. So, by the definition of $\mu_m$, we obtain
\begin{eqnarray*}
\frac{\left|D^{\alpha}T_{m,n}(\varphi)(x)-D^{\alpha}T_n(\varphi)(x)\right|}{T_{\alpha}M_{\alpha}}\leq \frac{c_0t_1M_1\sqrt{d}\|\theta\|_{(t'_j/2)}\|\varphi\|_{(t'_j/2)}}{2m}.
\end{eqnarray*}
Now, there exists $C>0$ such that $\ds \sup_{\varphi\in B}\|\varphi\|_{(t'_j/2)}\leq C$. If we take large enough $m_0$, such that $1/m_0\leq 2\varepsilon/\left(c_0Ct_1M_1\sqrt{d}\|\theta\|_{(t'_j/2)}\right)$, then, for all $m\geq m_0$, $T_{m,n}-T_n\in\mathcal{M}(B,V)$.

Now, we  prove that $T_n\longrightarrow \mathrm{Id}$ in $\mathcal{L}_c\left(\dot{\tilde{\mathcal{B}}}^{\{M_p\}}\left(\RR^d\right), \dot{\tilde{\mathcal{B}}}^{\{M_p\}}\left(\RR^d\right)\right)$. Let $B$, $V$ and $\mathcal{M}(B,V)$ be the same as above. There exists $C>0$ such that $\|\varphi\|_{(t_j/2)}\leq C$, for all $\varphi\in B$. Moreover, by lemma \ref{120}, for the chosen $\varepsilon$ and $(t_j)$, there exists $K\subset\subset\RR^d$ such that $\ds \frac{\left|D^{\alpha}\varphi(x)\right|}{T_{\alpha}M_{\alpha}}\leq \frac{\varepsilon}{2\left(1+\|\theta\|_{L^{\infty}}\right)}$ for all $\alpha\in\NN^d$, $x\in\RR^d\backslash K$ and $\varphi\in B$. There exists $n_0$ such that $K\subset\subset \mathrm{int}K_{n_0}$ and $C\|\theta\|_{(t_j/2)}/2^{n_0}\leq \varepsilon/2$. So, for $n\geq n_0$, we have\\
$\ds\frac{\left|D^{\alpha}T_n(\varphi)(x)-D^{\alpha}\varphi(x)\right|}{T_{\alpha}M_{\alpha}}$
\begin{eqnarray*}
&\leq& \left|1-\theta(x/2^n)\right|\frac{\left|D^{\alpha}\varphi(x)\right|}{T_{\alpha}M_{\alpha}}+
\sum_{\substack{\beta\leq\alpha\\ \beta\neq 0}}{\alpha\choose\beta}\frac{\left|D^{\beta}\theta(x/2^n)\right| \left|D^{\alpha-\beta}\varphi(x)\right|}{2^{n|\beta|}T_{\alpha}M_{\alpha}}\\
&\leq&\frac{\varepsilon}{2}+\frac{\|\theta\|_{(t_j/2)}\|\varphi\|_{(t_j/2)}}{2^n}\leq \varepsilon,
\end{eqnarray*}
that is $T_n-\mathrm{Id}\in \mathcal{M}(B,V)$, for all $n\geq n_0$.  Thus, $\mathrm{Id}$ belongs to the sequential closure of $\left(\dot{\tilde{\mathcal{B}}}^{\{M_p\}}\right)'\left(\RR^d\right)\otimes \dot{\tilde{\mathcal{B}}}^{\{M_p\}}\left(\RR^d\right)$.\qed
\end{proof}

If $E$ is a complete l.c.s., by proposition \ref{pr11}, proposition 1.4. of \cite{Komatsu3} and (\ref{1110}), we have the following isomorphisms of l.c.s.
\begin{equation}\label{incl}
\dot{\tilde{\mathcal{B}}}^{\{M_p\}}\left(\RR^d\right)\varepsilon E\cong \mathcal{L}_{\epsilon}\left(\left(\dot{\tilde{\mathcal{B}}}^{\{M_p\}}\right)'_c\left(\RR^d\right), E\right)\cong \mathcal{L}_{\epsilon}\left(E'_c, \dot{\tilde{\mathcal{B}}}^{\{M_p\}}\left(\RR^d\right)\right)\cong \dot{\tilde{\mathcal{B}}}^{\{M_p\}}\hat{\otimes}_{\epsilon} E.
\end{equation}

Let $E$ be a complete l.c.s. Define the space $\dot{\tilde{\mathcal{B}}}^{\{M_p\}}\left(\RR^d;E\right)$ as the space of all smooth $E-$valued functions $\varphi$ on $\RR^d$
 so that
\begin{itemize}
\item[$i)$] for each continuous seminorm $q$ of $E$ and $(t_j)\in\mathfrak{R}$ there exists $C>0$ such that
$\ds q_{(t_j)}(\varphi)=\sup_{\alpha\in\NN^d}\sup_{x\in\RR^d} q\left(\frac{D^{\alpha}\varphi(x)}{T_{\alpha}M_{\alpha}}\right)$,
\item[$ii)$] for every $\varepsilon>0$, $(t_j)\in\mathfrak{R}$ and $q$ a continuous seminorm on $E$, there exists $K\subset\subset\RR^d$ such that $\ds q\left(\frac{D^{\alpha}\varphi(x)}{T_{\alpha}M_{\alpha}}\right)\leq \varepsilon$, for all $\alpha\in\NN^d$ and $x\in \RR^d\backslash K$.
\end{itemize}
We equip $\dot{\tilde{\mathcal{B}}}^{\{M_p\}}\left(\RR^d;E\right)$ with the locally convex topology generated by  seminorms
$ q_{(t_j)}$, $q$ are seminorms on $E$ and $(t_j)\in\mathfrak{R}$. This topology is obviously Hausdorff and hence, $\dot{\tilde{\mathcal{B}}}^{\{M_p\}}\left(\RR^d;E\right)$ is a l.c.s.

\begin{proposition}\label{130}
$\dot{\tilde{\mathcal{B}}}^{\{M_p\}}\left(\RR^d;E\right)$  and $\dot{\tilde{\mathcal{B}}}^{\{M_p\}}\left(\RR^d\right)\varepsilon E$,
 are isomorphic l.c.s.
\end{proposition}

\begin{proof} By (\ref{incl}), it is enough to prove  that $\dot{\tilde{\mathcal{B}}}^{\{M_p\}}\left(\RR^d;E\right)\cong\mathcal{L}_{\epsilon}\left(E'_c,
\dot{\tilde{\mathcal{B}}}^{\{M_p\}}\left(\RR^d\right)\right)$. Let $\varphi\in \dot{\tilde{\mathcal{B}}}^{\{M_p\}}\left(\RR^d;E\right)$, $e'\in E'$ and
$\tilde{\varphi}_{e'}(x)=\langle e',\varphi(x)\rangle, x\in \RR^d$. Clearly, $\tilde{\varphi}_{e'}$ is smooth
 and $D^{\alpha}\tilde{\varphi}_{e'}=\left\langle e', D^{\alpha}\varphi\right\rangle$.  Let $(t_j)\in \mathfrak{R}$ and $\varepsilon>0.$ Then
\begin{eqnarray*}
\frac{\left|D^{\alpha}\tilde{\varphi}_{e'}(x)\right|}{T_{\alpha}M_{\alpha}}=\left|\left\langle e',
\frac{D^{\alpha}\varphi(x)}{T_{\alpha}M_{\alpha}}\right\rangle\right|\leq C_1q\left(\frac{D^{\alpha}\varphi(x)}{T_{\alpha}M_{\alpha}}\right)\leq C_1q_{(t_j)}(\varphi),
\alpha\in\NN^d, x\in \RR^d,
\end{eqnarray*}
and  there exists $K\subset\subset\RR^d$ such that $q\left(D^{\alpha}\varphi(x)/(T_{\alpha}M_{\alpha})\right)\leq \varepsilon/C_1$, for all $\alpha\in\NN^d$
and $x\in \RR^d\backslash K$. Similarly as above, one obtains that $\left|D^{\alpha}\tilde{\varphi}_{e'}(x)\right|/(T_{\alpha}M_{\alpha})\leq \varepsilon$
for all $\alpha\in \NN^d$ and $x\in\RR^d\backslash K$, i.e. $\tilde{\varphi}_{e'}\in \dot{\tilde{\mathcal{B}}}^{\{M_p\}}\left(\RR^d\right)$.
Let $\varphi\in \dot{\tilde{\mathcal{B}}}^{\{M_p\}}\left(\RR^d;E\right)$  and observe the mapping
$T_{\varphi}: E'\longrightarrow \dot{\tilde{\mathcal{B}}}^{\{M_p\}}\left(\RR^d\right)$, $e'\mapsto T_{\varphi}(e')=\tilde{\varphi}_{e'}$.\\
We prove that $T_{\varphi}\in \mathcal{L}\left(E'_c, \dot{\tilde{\mathcal{B}}}^{\{M_p\}}\left(\RR^d\right)\right)$.
Let
$ A=\left\{\frac{D^{\alpha}\varphi(x)}{(T_{\alpha}M_{\alpha}}\Big|x\in\RR^d,\, \alpha\in\NN^d\right\}$.
We will prove that $A$ is precompact in $E$. Let $U=\{e\in E|q_1(e)\leq r,...,q_n(e)\leq r\}$ be a neighborhood of zero in $E$.
 For the chosen $r$, $(t_j)$ and $q_1,...,q_n$, there exists $K\subset\subset\RR^d$ such that
 $q_l\left(D^{\alpha}\varphi(x)/(T_{\alpha}M_{\alpha})\right)\leq \ds\frac{r}{2}$, for all $\alpha\in\NN^d$, $x\in\RR^d\backslash K$ and $l=1,...,n$. Moreover, there exists $C>0$ such that $q_{l,(t_j/2)}(\varphi)\leq C$, for all $l=1,...,n$. Take $s\in\ZZ_+$ such that $1/2^s\leq r/(2C)$. Then, if $|\alpha|\geq s$, we have $q_l\left(D^{\alpha}\varphi(x)/(T_{\alpha}M_{\alpha})\right)\leq \ds\frac{r}{2}$ for all $x\in\RR^d$. The set $A'=\left\{ D^{\alpha}\varphi(x)/(T_{\alpha}M_{\alpha})\Big|x\in K,\, |\alpha|< s\right\}$ is obviously compact in $E$. So, there exists a finite subset $B'_0$ of $A'$ such that $A'\subseteq B'_0+U$. Take $x_1\in K$, $x_2\in\RR^d\backslash K$ and let $\beta\in\NN^d$ be a fixed $d$-tuple such that $|\beta|>s$. Consider the set $\ds B_0=B'_0\bigcup \left\{D^{\beta}\varphi(x_1)/(T_{\beta}M_{\beta}),\varphi(x_2)\right\}\subseteq A$. If $|\alpha|<s$ and $x\in K$,
  $D^{\alpha}\varphi(x)/(T_{\alpha}M_{\alpha})\in B_0+U$. If $|\alpha|\geq s$ and $x\in K$, we have
\begin{eqnarray*}
q_l\left(\frac{D^{\alpha}\varphi(x)}{T_{\alpha}M_{\alpha}}-\frac{D^{\beta}\varphi(x_1)}{T_{\beta}M_{\beta}}\right)\leq q_l\left(\frac{D^{\alpha}\varphi(x)}{T_{\alpha}M_{\alpha}}\right)+ q_l\left(\frac{D^{\beta}\varphi(x_1)}{T_{\beta}M_{\beta}}\right)\leq r,
\, l=1,...,n.
\end{eqnarray*}
Also,  if $x\in\RR^d\backslash K$ and $\alpha\in\NN^d$, we have
\begin{eqnarray*}
q_l\left(\frac{D^{\alpha}\varphi(x)}{T_{\alpha}M_{\alpha}}-\varphi(x_2)\right)\leq q_l\left(\frac{D^{\alpha}\varphi(x)}{T_{\alpha}M_{\alpha}}\right)+q_l\left(\varphi(x_2)\right)\leq r,
\, l=1,...,n.
\end{eqnarray*}
We obtain that $A\subseteq B_0+U$. Thus, $A$ is precompact.

Let $V=\left\{\psi\in\dot{\tilde{\mathcal{B}}}^{\{M_p\}}\left(\RR^d\right)\big|\|\psi\|_{(t_j)}\leq \varepsilon\right\}$ be a neighborhood of zero in $\dot{\tilde{\mathcal{B}}}^{\{M_p\}}$. Because $A$ is precompact and $E$ is complete l.c.s., $\tilde A$ - the closed convex circled hull of $A$ is compact.
 Let $W=\left(1/\varepsilon \tilde{A}\right)^{\circ}$ ($ ^{\circ}$ means  the polar).  Let $e'\in W$. Then
\begin{eqnarray*}
\frac{\left|D^{\alpha}T_{\varphi}(e')(x)\right|}{T_{\alpha}M_{\alpha}}=\left|\left\langle e', \frac{D^{\alpha}\varphi(x)}{T_{\alpha}M_{\alpha}}\right\rangle\right|\leq \varepsilon,\,
\alpha\in\NN^d, x\in\RR^d,
\end{eqnarray*}
and the continuity of $T_{\varphi}$ follows.

Now we  prove that the topology of $\dot{\tilde{\mathcal{B}}}^{\{M_p\}}\left(\RR^d;E\right)$ is the induced topology from
 $\mathcal{L}_{\epsilon}\left(E'_c,\dot{\tilde{\mathcal{B}}}^{\{M_p\}}\left(\RR^d\right)\right)$ when we consider
 it as a subspace of the latter by the injection $\varphi\mapsto T_{\varphi}$. Let $\mathcal{M}(B,V)$ be a neighborhood of zero
 in $\mathcal{L}_{\epsilon}\left(E'_c, \dot{\tilde{\mathcal{B}}}^{\{M_p\}}\left(\RR^d\right)\right)$, where $V$ is as above
  and $B$ is an equicontinuous subset of $E'$. Let $U=\{e\in E|q_1(e)\leq r,...,q_n(e)\leq r\}$ be a neighborhood of zero in $E$
   such that $|\langle e',e\rangle|\leq \varepsilon$, when $e\in U$ and $e'\in B$. Let
   $$W=\left\{\varphi\in\dot{\tilde{\mathcal{B}}}^{\{M_p\}}\left(\RR^d;E\right)\big|q_{1,(t_j)}(\varphi)\leq r,..., q_{n,(t_j)}(\varphi)\leq r\right\}.$$
   Then, for $\varphi\in W$, $\ds \frac{D^{\alpha}\varphi(x)}{T_{\alpha}M_{\alpha}}\in U$ for all $\alpha\in\NN^d$ and $x\in \RR^d$. Hence, for $e'\in B$,
$
\left|D^{\alpha}T_{\varphi}(e')(x)\right|/(T_{\alpha}M_{\alpha})\leq \varepsilon,\,
\alpha\in\NN^d, x\in\RR^d,
$
i.e. $T_{\varphi}\in \mathcal{M}(B,V)$, for all $\varphi\in W$. Conversely, let $W$ be a neighborhood of zero in
$\dot{\tilde{\mathcal{B}}}^{\{M_p\}}\left(\RR^d;E\right)$ given as above. Consider  $U$ as above and
$B=U^{\circ}$. If $\varphi\in W$ and $e'\in B$, then $\|T_{\varphi}(e')\|_{(t_j)}\leq 1$.
Let $V=\left\{\psi\in\dot{\tilde{\mathcal{B}}}^{\{M_p\}}\left(\RR^d;E\right)|\|\psi\|_{(t_j)}\leq 1\right\}$ and
$\tilde{G}=\mathcal{M}(B,V)\cap \left\{T_{\varphi}|\varphi\in\dot{\tilde{\mathcal{B}}}^{\{M_p\}}\left(\RR^d;E\right)\right\}$.
 Let $T_{\varphi}\in\tilde{G}$. Then, for all $e'\in B$, $T_{\varphi}(e')\in V$, i.e. $\ds \left\|T_{\varphi}(e')\right\|_{(t_j)}\leq 1$. So, we have
\begin{eqnarray*}
\left|\left\langle e',\frac{D^{\alpha}\varphi(x)}{T_{\alpha}M_{\alpha}}\right\rangle\right|= \frac{\left|D^{\alpha}T_{\varphi}(e')(x)\right|}{T_{\alpha}M_{\alpha}}\leq 1,\,
\alpha\in\NN^d, x\in\RR^d, e'\in B.
\end{eqnarray*}
We obtain that, for all $\alpha\in\NN^d$ and $x\in\RR^d$, $\ds \frac{D^{\alpha}\varphi(x)}{T_{\alpha}M_{\alpha}}\in B^{\circ}=
U^{\circ\circ}=U$. But this means that  $\varphi\in W$. Hence, we proved that $\varphi\mapsto T_{\varphi}$ is a topological imbedding of $\dot{\tilde{\mathcal{B}}}^{\{M_p\}}\left(\RR^d;E\right)$ into $\mathcal{L}_{\epsilon}\left(E'_c, \dot{\tilde{\mathcal{B}}}^{\{M_p\}}\left(\RR^d\right)\right)$. It remains to prove that this mapping
is a surjection. By theorem 1.12 of \cite{Komatsu3}, $\dot{\tilde{\mathcal{B}}}^{\{M_p\}}\left(\RR^d\right)\varepsilon E\cong \mathcal{L}_{\epsilon}\left(E'_c, \dot{\tilde{\mathcal{B}}}^{\{M_p\}}\left(\RR^d\right)\right)$ is identified with the space of all $f\in \mathcal{C}\left(\RR^d;E\right)$ such that:
\begin{itemize}
\item[$i)$] for any $e'\in E'$, the function $\langle e', f(\cdot)\rangle$ is in $\dot{\tilde{\mathcal{B}}}^{\{M_p\}}\left(\RR^d\right)$;
\item[$ii)$] for every equicontinuous set $A'$ in $E'$, the set $\{\langle e', f(\cdot)\rangle| e'\in A'\}$ is relatively compact in $\dot{\tilde{\mathcal{B}}}^{\{M_p\}}\left(\RR^d\right)$.
\end{itemize}
Every such $f$ generates an operator $L'\in \mathcal{L}\left(E'_c,\dot{\tilde{\mathcal{B}}}^{\{M_p\}}\left(\RR^d\right)\right)$ by
$L'(e')(\cdot)=\tilde f_{e'}=\langle e', f(\cdot)\rangle$, which gives the algebraic isomorphism between the space of all $f\in \mathcal{C}\left(\RR^d;E\right)$ which satisfy the above conditions and $\mathcal{L}\left(E'_c, \dot{\tilde{\mathcal{B}}}^{\{M_p\}}\left(\RR^d\right)\right)$. We will prove that every such $f$ belongs to $\dot{\tilde{\mathcal{B}}}^{\{M_p\}}\left(\RR^d;E\right)$ and obtain the desired surjectivity. So, let $f\in \mathcal{C}\left(\RR^d;E\right)$ be a function that satisfies the conditions $i)$ and $ii)$.
By the above conditions, $\tilde{f}_{e'}\in \dot{\tilde{\mathcal{B}}}^{\{M_p\}}\left(\RR^d\right)\subseteq\EE^{\{M_p\}}\left(\RR^d\right)$, so,
by theorem 3.10 of \cite{Komatsu3}, we get that $f\in\EE^{\{M_p\}}\left(\RR^d; E\right)$. Hence $f$ is smooth $E$-valued  and
from the quoted theorem it follows that $D^{\alpha}\tilde{f}_{e'}(x)=\left\langle e', D^{\alpha}f(x)\right\rangle$. Let $(t_j)\in\mathfrak{R}$.
Then $$\ds\left|\left\langle e',\frac{D^{\alpha}f(x)}{T_{\alpha}M_{\alpha}}\right\rangle\right|= \frac{\left|D^{\alpha}\tilde{f}_{e'}(x)\right|}{T_{\alpha}
M_{\alpha}}\leq \left\|\tilde{f}_{e'}\right\|_{(t_j)}.$$ So, the set $\ds \left\{\frac{D^{\alpha}f(x)}{T_{\alpha}M_{\alpha}}\Big|\alpha\in\NN^d,\, x\in\RR^d\right\}$
is weakly bounded, hence it is bounded in $E$.
Let $q$ be a continuous seminorm in $E$ and $U=\{e\in E| q(e)\leq \varepsilon\}$ in $E$. There exists $C>0$ such that $\ds q\left(D^{\alpha}f(x)/(T_{\alpha}M_{\alpha})\right)\leq C$, for all $\alpha\in\NN^d$ and $x\in\RR^d$. Since $A'=W^{\circ}$ is equicontinuous set in $E'$, $\left\{\tilde{f}_{e'}|e'\in A'\right\}$ is relatively compact in $\dot{\tilde{\mathcal{B}}}^{\{M_p\}}\left(\RR^d\right)$. By lemma \ref{120}, for the chosen $(t_j)$, there exists $K\subset\subset\RR^d$ such that
$\ds \left|D^{\alpha}\tilde{f}_{e'}(x)\right|/(T_{\alpha}M_{\alpha})\leq 1$, for all $\alpha \in\NN^d$, $x\in\RR^d\backslash K$ and $e'\in A'$.
We obtain that, for $\alpha\in\NN^d$ and $x\in\RR^d\backslash K$, $\ds \frac{D^{\alpha}f(x)}{T_{\alpha}M_{\alpha}}\in A'^{\circ}=U^{\circ\circ}=U$.
But then, $\ds q\left(D^{\alpha}f(x)/(T_{\alpha}M_{\alpha})\right)\leq \varepsilon$, for all $\alpha\in \NN^d$ and $x\in\RR^d\backslash K$. We obtain that $f\in \dot{\tilde{\mathcal{B}}}^{\{M_p\}}\left(\RR^d;E\right)$.\qed
\end{proof}

Hence, if we take $E=\dot{\tilde{\mathcal{B}}}^{\{M_p\}}\left(\RR^m\right)$, we get,
\begin{equation} \label{333}
\dot{\tilde{\mathcal{B}}}^{\{M_p\}}\left(\RR^d;\dot{\tilde{\mathcal{B}}}^{\{M_p\}}\left(\RR^m\right)\right)\cong \dot{\tilde{\mathcal{B}}}^{\{M_p\}}\left(\RR^d\right)\varepsilon \dot{\tilde{\mathcal{B}}}^{\{M_p\}}\left(\RR^m\right)\cong \dot{\tilde{\mathcal{B}}}^{\{M_p\}}\left(\RR^d\right)\hat{\otimes}_{\epsilon} \dot{\tilde{\mathcal{B}}}^{\{M_p\}}\left(\RR^m\right).
\end{equation}

\begin{proposition}\label{100}
$\dot{\tilde{\mathcal{B}}}^{\{M_p\}}\left(\RR^{d_1+d_2}\right)\cong \dot{\tilde{\mathcal{B}}}^{\{M_p\}}\left(\RR^{d_1}\right)\hat{\otimes}_{\epsilon} \dot{\tilde{\mathcal{B}}}^{\{M_p\}}\left(\RR^{d_2}\right)$.
\end{proposition}

\begin{proof}  By (\ref{333})
it is enough to prove $\dot{\tilde{\mathcal{B}}}^{\{M_p\}}\left(\RR^{d_1+d_2}\right)\cong \dot{\tilde{\mathcal{B}}}^{\{M_p\}}\left(\RR^{d_1};
\dot{\tilde{\mathcal{B}}}^{\{M_p\}}\left(\RR^{d_2}\right)\right)$. Let $f\in\dot{\tilde{\mathcal{B}}}^{\{M_p\}}\left(\RR^{d_1}; \dot{\tilde{\mathcal{B}}}^{\{M_p\}}\left(\RR^{d_2}\right)\right)$.
Put $\varphi(x,y)=f(x)(y)$, $x\in\RR^{d_1}, y\in\RR^{d_2}$. One can prove in a standard way that $\varphi$ is smooth on $\RR^{d_1+d_2}$ and that
$D^{\alpha}_x D^{\beta}_y\varphi(x,y)=D^{\beta}_y\left(D^{\alpha}_x f(x)\right)(y)$ for all $\alpha\in\NN^{d_1}$, $\beta\in\NN^{d_2}$ and $(x,y)\in\RR^{d_1+d_2}$. Let $(t_j)\in\mathfrak{R}$ and $\alpha\in\NN^{d_1}$, $\beta\in\NN^{d_2}$ and $(x,y)\in\RR^{d_1+d_2}$. Then
\begin{eqnarray*}
\frac{\left|D^{\alpha}_x D^{\beta}_y \varphi(x,y)\right|}{T_{\alpha+\beta}M_{\alpha+\beta}}\leq \left\|\frac{D^{\alpha}f(x)}{T_{\alpha}M_{\alpha}}\right\|_{(t_j)}\leq \sup_{\alpha\in \NN^{d_1}} \sup_{x\in\RR^{d_1}}\left\|\frac{D^{\alpha}f(x)}{T_{\alpha}M_{\alpha}}\right\|_{(t_j)},
\end{eqnarray*}
which is a seminorm in $\dot{\tilde{\mathcal{B}}}^{\{M_p\}}\left(\RR^{d_1}; \dot{\tilde{\mathcal{B}}}^{\{M_p\}}\left(\RR^{d_2}\right)\right)$.
Moreover, if $\varepsilon>0$, then there exists $K_1\subset\subset\RR^{d_1}$ such that
$\ds \sup_{\alpha\in\NN^{d_1}}\sup_{x\in\RR^{d_1}\backslash K_1} \left\|\frac{D^{\alpha}f(x)}{T_{\alpha}M_{\alpha}}\right\|_{(t_j)}\leq \varepsilon$.
In the proof of proposition \ref{130} we proved that $\ds A=\left\{\frac{D^{\alpha}f(x)}{T_{\alpha}M_{\alpha}}\Big|\alpha\in\NN^{d_1},\, x\in\RR^{d_1}\right\}$ is a
precompact subset of $\dot{\tilde{\mathcal{B}}}^{\{M_p\}}\left(\RR^{d_2}\right)$. So, by lemma \ref{120}, for the chosen $(t_j)$ and $\varepsilon$,
there exists $K_2\subset\subset\RR^{d_2}$such that $$\ds \frac{\left|D^{\beta}_y\left(D^{\alpha}_x f(x)\right)(y)\right|}{T_{\alpha}T_{\beta}M_{\alpha}M_{\beta}}\leq \varepsilon, \alpha\in\NN^{d_1}, \beta\in\NN^{d_2},  x\in \RR^{d_1}, y\in
\RR^{d_2}\backslash K_2.$$ Then
$$\ds \frac{\left|D^{\alpha}_x D^{\beta}_y \varphi(x,y)\right|}{T_{\alpha+\beta}M_{\alpha+\beta}}\leq \varepsilon,  (x,y)\in\RR^{d_1+d_2}\backslash K, K=K_1\times K_2,
\alpha\in\NN^{d_1}, \beta\in\NN^{d_2}.$$
Hence, we obtained that $\varphi\in\dot{\tilde{\mathcal{B}}}^{\{M_p\}}\left(\RR^{d_1+d_2}\right)$. and that
the injection $$f\mapsto \varphi,\, \dot{\tilde{\mathcal{B}}}^{\{M_p\}}\left(\RR^{d_1}; \dot{\tilde{\mathcal{B}}}^{\{M_p\}}\left(\RR^{d_2}\right)\right)\longrightarrow \dot{\tilde{\mathcal{B}}}^{\{M_p\}}\left(\RR^{d_1+d_2}\right)$$ is continuous.

Now, let $\varphi\in\dot{\tilde{\mathcal{B}}}^{\{M_p\}}\left(\RR^{d_1+d_2}\right)$. Let $f$ be the mapping $x\mapsto f(x)= \varphi(x,\cdot)$, $\RR^{d_1}\longrightarrow \dot{\tilde{\mathcal{B}}}^{\{M_p\}}\left(\RR^{d_2}\right)$. Again, by the standard arguments we have that $f$  is a smooth mapping of $\RR^{d_1}$ into
$\dot{\tilde{\mathcal{B}}}^{\{M_p\}}\left(\RR^{d_2}\right)$.  Moreover, $D^{\alpha}f(x)=D^{\alpha}_x \varphi(x,\cdot)$. Let $(t_j),(\tilde{t}_j)\in\mathfrak{R}$. By lemma \ref{3}, we can choose $(t'_j)\in\mathfrak{R}$ such that $t'_j\leq t_j$, $t'_j\leq \tilde{t}_j$ and $T'_{j+k}\leq 2^{j+k}T'_j T'_k$, for all $j,k\in\NN$. Because
\begin{eqnarray*}
\frac{\left|D^{\alpha}_x D^{\beta}_y \varphi(x,y)\right|}{T_{\alpha}\tilde{T}_{\beta}M_{\alpha}M_{\beta}}\leq \frac{c_0(2H)^{|\alpha|+|\beta|}\left|D^{\alpha}_x D^{\beta}_y \varphi(x,y)\right|}{T'_{\alpha+\beta}M_{\alpha+\beta}}\leq c_0\|\varphi\|_{(t'_j/(2H))},
\end{eqnarray*}
for all $\alpha\in\NN^{d_1}$, $\beta\in\NN^{d_2}$ and $(x,y)\in\RR^{d_1+d_2}$, we get
$$\ds \sup_{\alpha\in\NN^{d_1}}\sup_{x\in\RR^{d_1}} \left\|\frac{D^{\alpha}f(x)}{T_{\alpha}M_{\alpha}}\right\|_{(\tilde{t}_j)} \leq c_0 \|\varphi\|_{(t'_j/(2H))}.$$ Let $(t_j),(\tilde{t}_j)\in\mathfrak{R}$, $\varepsilon>0$ be fixed and choose $(t'_j)\in\mathfrak{R}$ as above. Denote $t''_j=t'_j/(2H)$. Then there exists  $K\subset\subset\RR^{d_1+d_2}$ such that
$$\ds \frac{\left|D^{\alpha}_x D^{\beta}_y \varphi(x,y)\right|}{T''_{\alpha+\beta}M_{\alpha+\beta}}\leq \frac{\varepsilon}{c_0}, \alpha\in\NN^{d_1}, \beta\in\NN^{d_2}, (x,y)\in\RR^{d_1+d_2}\backslash K.$$ Let $K_1$ be the projection of $K$ on $\RR^{d_1}$. Then $K_1$ is a compact subset of $\RR^{d_1}$ and if $x\in\RR^{d_1}\backslash K_1$ is fixed, by the above estimates, we have that
$\ds\left\|\frac{D^{\alpha}f(x)}{T_{\alpha}M_{\alpha}}\right\|_{(\tilde{t}_j)}\leq \varepsilon$, for all $\alpha\in \NN^{d_1}$. Because $x\in \RR^{d_1}\backslash K_1$ is arbitrary, it follows that $f\in\dot{\tilde{\mathcal{B}}}^{\{M_p\}}\left(\RR^{d_1}; \dot{\tilde{\mathcal{B}}}^{\{M_p\}}\left(\RR^{d_2}\right)\right)$. From the above estimates, it follows that the mapping $\varphi\mapsto f$, $\dot{\tilde{\mathcal{B}}}^{\{M_p\}}\left(\RR^{d_1+d_2}\right)\longrightarrow \dot{\tilde{\mathcal{B}}}^{\{M_p\}}\left(\RR^{d_1}; \dot{\tilde{\mathcal{B}}}^{\{M_p\}}\left(\RR^{d_2}\right)\right)$, which is obviously injection, is continuous. Observe that the composition in both directions of the two mappings defined above is the identity mapping. So $\dot{\tilde{\mathcal{B}}}^{\{M_p\}}\left(\RR^{d_1+d_2}\right)\cong \dot{\tilde{\mathcal{B}}}^{\{M_p\}}\left(\RR^{d_1}; \dot{\tilde{\mathcal{B}}}^{\{M_p\}}\left(\RR^{d_2}\right)\right)$.\qed
\end{proof}

\section{Existence of convolution of two Roumieu ultradistributions}

We follow in this subsection the ideas for the convolution of Schwartz distributions but since in our case the topological properties are more delicate, the proofs are adequately
more complicate.

We  define an alternative l.c. topology on $\tilde{\DD}^{\{M_p\}}_{L^{\infty}}$ such that its dual is algebraically isomorphic to $\tilde{\DD}'^{\{M_p\}}_{L^1}$ (c.f. \cite{Wagner} for the case of Schwartz distributions).
 Let $g\in \mathcal{C}_0\left(\RR^d\right)$ (the space of all continuous functions that vanish at infinity) and $(t_j)\in\mathfrak{R}$. The seminorms
\begin{eqnarray*}
p_{g,(t_j)}(\varphi)=\sup_{\alpha\in\NN^d}\sup_{x\in\RR^d} \frac{\left|g(x)D^{\alpha}\varphi(x)\right|}{T_{\alpha}M_{\alpha}},\, \varphi\in\tilde{\DD}^{\{M_p\}}_{L^{\infty}}
\end{eqnarray*}
generate l.c. topology on $\tilde{\DD}^{\{M_p\}}_{L^{\infty}}$ and this space with this topology is denoted by $\tilde{\tilde{\DD}}^{\{M_p\}}_{L^{\infty}}$. Note that the inclusions $\tilde{\DD}^{\{M_p\}}_{L^{\infty}}\longrightarrow \tilde{\tilde{\DD}}^{\{M_p\}}_{L^{\infty}}$ and $\DD^{\{M_p\}}\longrightarrow\tilde{\tilde{\DD}}^{\{M_p\}}_{L^{\infty}}\longrightarrow\EE^{\{M_p\}}$ are continuous.

\begin{lemma}\label{10}
Let $P(D)=\sum_{\alpha}c_{\alpha}D^{\alpha}$ be an ultradifferential operator of class $\{M_p\}$. Then $P(D)$ is a continuous mapping from $\tilde{\tilde{\DD}}^{\{M_p\}}_{L^{\infty}}$ to $\tilde{\tilde{\DD}}^{\{M_p\}}_{L^{\infty}}$.
\end{lemma}

\begin{proof} We know that $c_{\alpha}$ are constants such that for every $L>0$ there exists $C>0$ such that
$\ds\sup_{\alpha}|c_{\alpha}|M_{\alpha}/L^{|\alpha|}\leq C$. So, by lemma 3.4 of \cite{Komatsu3}, there exists $(r_j)\in\mathfrak{R}$
 and $C_1>0$ such that $\ds\sup_{\alpha}|c_{\alpha}|R_{\alpha}M_{\alpha}\leq C_1$. Let $g\in \mathcal{C}_0$ and $(t_j)\in\mathfrak{R}$. Take $(s'_j)\in\mathfrak{R}$ such that $s'_j\leq r_j$ and $s'_j\leq t_j$ ($S_k\leq T_k, S_k\leq R_k$). By lemma \ref{3}, there exists $(s_j)\in\mathfrak{R}$ such that $s_j\leq s'_j$ and $S_{j+k}\leq 2^{j+k}S_jS_k$, for all $j,k\in\NN$. Then, for $\varphi\in\tilde{\tilde{\DD}}^{\{M_p\}}_{L^{\infty}}$, we have\\
$\ds\frac{\left|g(x)D^{\alpha}\left(P(D)\varphi(x)\right)\right|}{T_{\alpha}M_{\alpha}}$
\begin{eqnarray*}
&\leq& \sum_{\beta}\frac{|c_{\beta}|\left|g(x)D^{\alpha+\beta}\varphi(x)\right|}{T_{\alpha}M_{\alpha}}
\leq C_1p_{g,(s_j/(4H))}(\varphi)\sum_{\beta}\frac{S_{\alpha+\beta}M_{\alpha+\beta}}
{(4H)^{|\alpha|+|\beta|}T_{\alpha}R_{\beta}M_{\alpha}M_{\beta}}\\
&\leq& c_0C_1p_{g,(s_j/(4H))}(\varphi)\sum_{\beta}\frac{S_{\alpha}S_{\beta}}
{2^{|\alpha|+|\beta|}T_{\alpha}R_{\beta}}
\leq C_2p_{g,(s_j/(4H))}(\varphi), \alpha\in\NN^d, x\in\RR^d.
\end{eqnarray*}
Note that we can perform the same calculations as above without $g$, from what it will follow that $P(D)$ is well defined mapping from $\tilde{\tilde{\DD}}^{\{M_p\}}_{L^{\infty}}$ into $\tilde{\tilde{\DD}}^{\{M_p\}}_{L^{\infty}}$ and by the above, it is continuous.
\qed
\end{proof}
\indent Denote by $\left(\tilde{\tilde{\DD}}^{\{M_p\}}_{L^{\infty}}\right)'$ the strong dual of $\tilde{\tilde{\DD}}^{\{M_p\}}_{L^{\infty}}$. By the use of  cutoff functions one obtains the next lemma.
\begin{lemma}\label{15}
$\DD^{\{M_p\}}$ is sequentially dense in $\tilde{\tilde{\DD}}^{\{M_p\}}_{L^{\infty}}$. In particular, the inclusion $\left(\tilde{\tilde{\DD}}^{\{M_p\}}_{L^{\infty}}\right)'\longrightarrow\DD'^{\{M_p\}}$ is continuous.
\end{lemma}

\begin{lemma}\label{20}
The bilinear mapping $(\varphi,\psi)\mapsto\varphi\psi$, $\tilde{\tilde{\DD}}^{\{M_p\}}_{L^{\infty}}\times\tilde{\tilde{\DD}}^{\{M_p\}}_{L^{\infty}} \longrightarrow\tilde{\tilde{\DD}}^{\{M_p\}}_{L^{\infty}}$ is continuous.
\end{lemma}
\begin{proof}  Let $g\in \mathcal{C}_0$ and $(t_j)\in\mathfrak{R}$. Obviously, $\tilde{g}(x)=\sqrt{|g(x)|}\in \mathcal{C}_0$. Let $\varphi,\psi\in\tilde{\tilde{\DD}}^{\{M_p\}}_{L^{\infty}}$. Then
\begin{eqnarray*}
\frac{\left|g(x)D^{\alpha}\left(\varphi(x)\psi(x)\right)\right|}{2^\alpha T_{\alpha}M_{\alpha}}
&\leq&\frac{1}{2^\alpha} \sum_{\beta\leq\alpha}{\alpha\choose\beta} \frac{|g(x)|\left|D^{\beta}\varphi(x)
\right|\left|D^{\alpha-\beta}\psi(x)\right|}{T_{\alpha}M_{\alpha}}\\
&\leq& C
p_{\tilde{g},(t_j/2)}(\varphi)p_{\tilde{g},(t_j/2)}(\psi),\, x\in\RR^d, \alpha\in\NN^d.
\end{eqnarray*}
\qed
\end{proof}

\begin{proposition}\label{25}
The sets $\tilde{\DD}'^{\{M_p\}}_{L^1}$ and $\left(\tilde{\tilde{\DD}}^{\{M_p\}}_{L^{\infty}}\right)'$ are equal and the inclusion $\left(\tilde{\tilde{\DD}}^{\{M_p\}}_{L^{\infty}}\right)'\longrightarrow\tilde{\DD}'^{\{M_p\}}_{L^1}$ is continuous.
\end{proposition}

\begin{proof} Since, $\dot{\tilde{\mathcal{B}}}^{\{M_p\}}$ is continuously and densely injected in $\tilde{\tilde{\DD}}^{\{M_p\}}_{L^{\infty}}$,  it follows that the injection $\left(\tilde{\tilde{\DD}}^{\{M_p\}}_{L^{\infty}}\right)'\longrightarrow \tilde{\DD}'^{\{M_p\}}_{L^1}$ is continuous. Let $T\in \tilde{\DD}'^{\{M_p\}}_{L^1}$. Then, by theorem 1 of \cite{PilipovicT}, there exist an ultradifferential operator $P(D)$, of class $\{M_p\}$ and $F_1,F_2\in L^1$ such that $T=P(D)F_1+F_2$. Let $\varphi\in \DD^{\{M_p\}}$. Then
\begin{eqnarray*}
|\langle P(D)F_1,\varphi\rangle|=|\langle F_1,P(-D)\varphi\rangle|=\left|\int_{\RR^d}F_1(x)P(-D)\varphi(x)dx\right|.
\end{eqnarray*}
Because $F_1\in L^1\subseteq \mathcal{M}^1$ (integrable measures), by proposition 1.2.1. of \cite{Wagner}, there exists $g_1\in \mathcal{C}_0$ such that $\ds\left|\int_{\RR^d}F_1(x)f(x)dx\right|\leq \|fg_1\|_{L^{\infty}}$, for all $f\in\mathcal{BC}$ ($\mathcal{BC}$ is the space of continuous bounded functions on $\RR^d$). Let $(t_j)\in\mathfrak{R}$. We obtain, by  lemma \ref{10}, that for some $\tilde{g}_1\in \mathcal{C}_0$, $(t'_j)\in\mathfrak{R}$ and $C_1>0$,
\begin{eqnarray*}
|\langle P(D)F_1,\varphi\rangle|\leq \|g_1P(-D)\varphi\|_{L^{\infty}}\leq p_{g_1,(t_j)}(P(-D)\varphi)\leq C_1p_{\tilde{g}_1,(t'_j)}(\varphi).
\end{eqnarray*}
 Similarly, there exist $\tilde{g}_2\in \mathcal{C}_0$, $(t''_j)\in\mathfrak{R}$ and $C_2>0$ such that $|\langle F_2,\varphi\rangle|\leq C_2p_{\tilde{g}_2,(t''_j)}(\varphi)$, for all $\varphi\in\DD^{\{M_p\}}$. By lemma \ref{15}, $T\in\left(\tilde{\tilde{\DD}}^{\{M_p\}}_{L^{\infty}}\right)'$.\qed
\end{proof}

\begin{lemma}\label{27}
Let $S,T\in\DD'^{\{M_p\}}\left(\RR^d\right)$ are such that, for every $\varphi\in\DD^{\{M_p\}}\left(\RR^d\right)$, $(S\otimes T) \varphi^{\Delta}\in\tilde{\DD}'^{\{M_p\}}_{L^1}\left(\RR^{2d}\right)$. Then
$F:\DD^{\{M_p\}}\left(\RR^d\right)\longrightarrow \left(\tilde{\tilde{\DD}}^{\{M_p\}}_{L^{\infty}}\right)'\left(\RR^{2d}\right)$ defined by $F(\varphi)=(S\otimes T)\varphi^{\Delta}$ is linear and continuous.
\end{lemma}

\begin{proof} By proposition \ref{25}, $(S\otimes T) \varphi^{\Delta}\in\left(\tilde{\tilde{\DD}}^{\{M_p\}}_{L^{\infty}}\right)'\left(\RR^{2d}\right)$ for every $\varphi\in\DD^{\{M_p\}}\left(\RR^d\right)$. Because $\DD^{\{M_p\}}$ is bornologic, it is enough to prove that $F$ maps bounded sets into bounded sets. Let $B$ be a bounded set in $\DD^{\{M_p\}}\left(\RR^d\right)$. Then, there exist $K\subset\subset\RR^d$ and $h>0$ such that $B\subseteq\DD^{\{M_p\},h}_K$ and $B$ is bounded there. It is obvious that, without losing generality, we can assume that $K=K_{\RR^d}(0,q)$, for some $q>0$. Take $\chi\in\DD^{\{M_p\}}$ such that $\chi=1$ on $K$ and $0$ outside some bounded neighborhood of $K$. Then, for $\varphi\in B$ and $\psi\in\tilde{\tilde{\DD}}^{\{M_p\}}_{L^{\infty}}\left(\RR^{2d}\right)$, we have
\begin{eqnarray*}
\left\langle (S\otimes T)\varphi^{\Delta},\psi\right\rangle=\left\langle (S\otimes T) \chi^{\Delta}\varphi^{\Delta},\psi\right\rangle=
\left\langle (S\otimes T)\chi^{\Delta},\varphi^{\Delta}\psi\right\rangle,
\end{eqnarray*}
where, in the last equality, we used that $(S\otimes T)\chi^{\Delta}\in\left(\tilde{\tilde{\DD}}^{\{M_p\}}_{L^{\infty}}\right)'\left(\RR^{2d}\right)$ and $\varphi^{\Delta}\in \tilde{\tilde{\DD}}^{\{M_p\}}_{L^{\infty}}\left(\RR^{2d}\right)$ when $\varphi\in\DD^{\{M_p\}}\left(\RR^d\right)$. Let $\psi\in B_1$ for some bounded set $B_1$ in $\tilde{\tilde{\DD}}^{\{M_p\}}_{L^{\infty}}\left(\RR^{2d}\right)$. Let $g\in \mathcal{C}_0\left(\RR^{2d}\right)$ and $(t_j)\in\mathfrak{R}$. Then, for $\varphi\in B$ and $\psi\in B_1$, we have\\
$\ds \frac{\left|g(x,y)D^{\alpha}_x D^{\beta}_y \left(\varphi^{\Delta}(x,y)\psi(x,y)\right)\right|}{T_{\alpha+\beta}M_{\alpha+\beta}}$
\begin{eqnarray*}
&\leq& \sum_{\substack{\gamma\leq\alpha\\ \delta\leq\beta}}{\alpha\choose\gamma}{\beta\choose\delta} \frac{|g(x,y)|\left|D^{\gamma+\delta} \varphi(x+y)\right|\left|D^{\alpha-\gamma}_x D^{\beta-\delta}_y\psi(x,y)\right|}{T_{\alpha+\beta}M_{\alpha+\beta}}\\
&\leq&p_{K,h}(\varphi) p_{g,(t_j/2)}(\psi)\sum_{\substack{\gamma\leq\alpha\\ \delta\leq\beta}} {\alpha\choose\gamma}{\beta\choose\delta} \frac{(2h)^{|\gamma|+|\delta|}}{2^{|\alpha|+|\beta|}T_{\gamma+\delta}}\leq Cp_{K,h}(\varphi) p_{g,(t_j/2)}(\psi).
\end{eqnarray*}
Since $p_{K,h}(\varphi)$ and $p_{g,(t_j/2)}(\psi)$ are bounded when $\varphi\in B$ and $\psi\in B_1$,  the set\newline $\left\{\theta\in\tilde{\tilde{\DD}}^{\{M_p\}}_{L^{\infty}}\left(\RR^{2d}\right)|\theta=\varphi^{\Delta}\psi,\, \varphi\in B,\, \psi\in B_1\right\}$ is bounded in $\tilde{\tilde{\DD}}^{\{M_p\}}_{L^{\infty}}\left(\RR^{2d}\right)$. This implies that $\left\langle (S\otimes T) \varphi^{\Delta},\psi\right\rangle=\left\langle (S\otimes T)\chi^{\Delta},\varphi^{\Delta}\psi\right\rangle$ is bounded, for $\varphi\in B$ and $\psi\in B_1$. Thus, $F(\varphi)=(S\otimes T)\varphi^{\Delta}$, $\varphi\in B$ is bounded.\qed
\end{proof}

\begin{definition}
Let $S,T\in\DD'^{\{M_p\}}\left(\RR^d\right)$ are such that for every $\varphi\in\DD^{\{M_p\}}\left(\RR^d\right)$, $(S\otimes T)\varphi^{\Delta}\in\tilde{\DD}'^{\{M_p\}}_{L^1}\left(\RR^{2d}\right)$. Define the convolution of $S$ and $T$, $S*T\in\DD'^{\{M_p\}}\left(\RR^d\right)$, by
\begin{eqnarray*}
\langle S*T,\varphi\rangle={}_{\left(\tilde{\tilde{\DD}}^{\{M_p\}}_{L^{\infty}}\right)'}\left\langle(S\otimes T) \varphi^{\Delta},1\right\rangle_{\tilde{\tilde{\DD}}^{\{M_p\}}_{L^{\infty}}}; \, (1(x)=1\in \tilde{\tilde{\DD}}^{\{M_p\}}_{L^{\infty}}).
\end{eqnarray*}
\end{definition}
For every $a>0$, define the space $\dot{\mathcal{B}}^{\{M_p\}}_a=\left\{\varphi\in \dot{\tilde{\mathcal{B}}}^{\{M_p\}}\left(\RR^{2d}\right)|\mathrm{supp\,} \varphi\subseteq \Delta_a\right\}$, where $\Delta_a=\left\{(x,y)\in\RR^{2d}||x+y|\leq a\right\}$. With the seminorms $\|\varphi\|_{(t_j)} $ (now over $\RR^{2d}$), $\dot{\mathcal{B}}^{\{M_p\}}_a$ becomes a l.c.s. Define the space $\ds \dot{\mathcal{B}}^{\{M_p\}}_{\Delta}=\lim_{\substack{\longrightarrow\\ a\rightarrow\infty}}\dot{\mathcal{B}}^{\{M_p\}}_a$, where the inductive limit is strict; $\dot{\mathcal{B}}^{\{M_p\}}_{\Delta}$ is a l.c.s. because we have a continuous inclusion $\dot{\mathcal{B}}^{\{M_p\}}_{\Delta}\longrightarrow\EE^{\{M_p\}}$.

\begin{lemma} \label{123}
Let $a>0$. Then $\DD^{\{M_p\}}_{\Delta_a}\left(\RR^{2d}\right)=\left\{\varphi\in\DD^{\{M_p\}}\left(\RR^{2d}\right)|\mathrm{supp\,} \varphi\subseteq \Delta_a\right\}$ is sequentially dense in $\dot{\mathcal{B}}^{\{M_p\}}_a$.
\end{lemma}

\begin{proof} Let $\varphi\in \dot{\mathcal{B}}^{\{M_p\}}_a$. Take $\chi\in\DD^{\{M_p\}}\left(\RR^{2d}\right)$ such that $\chi(x,y)=1$ on $K_{\RR^{2d}}(0,1)$
and $\chi(x,y)=0$ out of $K_{\RR^{2d}}(0,2)$.  For $n\in\ZZ_+$, put $\chi_n(x,y)=\chi(x/n,y/n)$. Then $\varphi_n=\chi_n\varphi\in\DD^{\{M_p\}}_{\Delta_a}\left(\RR^{2d}\right)$ for all $n\in \ZZ_+$. Let $(t_j)\in\mathfrak{R}$. We have\\
$\ds\frac{\left|D^{\alpha}_x D^{\beta}_y\varphi(x,y)-D^{\alpha}_x D^{\beta}_y\varphi_n(x,y)\right|}{T_{\alpha+\beta}M_{\alpha+\beta}}$
\begin{eqnarray*}
&\leq& \left|1-\chi(x/n,y/n)\right|\frac{\left|D^{\alpha}_x D^{\beta}_y\varphi(x,y)\right|} {T_{\alpha+\beta}M_{\alpha+\beta}}\\
&{}&\,\,\,+\sum_{\substack{\gamma\leq\alpha\\ \delta\leq \beta\\ \gamma+\delta\neq 0}} {\alpha\choose\gamma}{\beta\choose\delta}\frac{\left|D^{\gamma}_x D^{\delta}_y \chi(x/n,y/n)\right| \left|D^{\alpha-\gamma}_x D^{\beta-\delta}_y\varphi(x,y)\right|} {n^{|\gamma|+|\delta|}T_{\alpha+\beta}M_{\alpha+\beta}}\\
&\leq&\left|1-\chi(x/n,y/n)\right|\frac{\left|D^{\alpha}_x D^{\beta}_y\varphi(x,y)\right|} {T_{\alpha+\beta}M_{\alpha+\beta}}+\frac{C_1C_2\|\varphi\|_{(t_j/2)}}{n}
\end{eqnarray*}
 By lemma \ref{5} and by the way we chose $\chi$, it follows that the above two terms tend to zero uniformly in $(x,y)\in\RR^{2d}$ and $\alpha,\beta\in\NN^d$ when $n\rightarrow\infty$.\qed
\end{proof}

\noindent Because $\ds\DD^{\{M_p\}}\left(\RR^{2d}\right)=\bigcup_{a\in\RR_+}\DD^{\{M_p\}}_{\Delta_a}\left(\RR^{2d}\right)$, by  lemma \ref{123}, it follows that $\DD^{\{M_p\}}\left(\RR^{2d}\right)$ is dense in $\dot{\mathcal{B}}^{\{M_p\}}_{\Delta}$. Moreover, one easily checks that the inclusions $\dot{\mathcal{B}}^{\{M_p\}}_{\Delta}\longrightarrow\EE^{\{M_p\}}\left(\RR^{2d}\right)$ and $\DD^{\{M_p\}}\left(\RR^{2d}\right)\longrightarrow\dot{\mathcal{B}}^{\{M_p\}}_{\Delta}$ are continuous, hence, the inclusion $\left(\dot{\mathcal{B}}^{\{M_p\}}_{\Delta}\right)'\longrightarrow\DD'^{\{M_p\}}\left(\RR^{2d}\right)$ is continuous ($\left(\dot{\mathcal{B}}^{\{M_p\}}_{\Delta}\right)'$ is the strong dual of $\dot{\mathcal{B}}^{\{M_p\}}_{\Delta}$).

\begin{theorem}
Let $S,T\in\DD'^{\{M_p\}}\left(\RR^d\right)$. The following statements are equivalent:
\begin{itemize}
\item[i)] the convolution of $S$ and $T$ exists;
\item[ii)] $S\otimes T\in \left(\dot{\mathcal{B}}^{\{M_p\}}_{\Delta}\right)'$;
\item[iii)] for all $\varphi\in\DD^{\{M_p\}}\left(\RR^d\right)$, $\left(\varphi*\check{S}\right)T\in\tilde{\DD}'^{\{M_p\}}_{L^1}\left(\RR^d\right)$ and for every compact subset $K$ of $\RR^d$,  $(\varphi,\chi)\mapsto \left\langle\left(\varphi*\check{S}\right)T,\chi\right\rangle$, $\DD^{\{M_p\}}_K\times\dot{\tilde{\mathcal{B}}}^{\{M_p\}}\left(\RR^d\right)\longrightarrow \CC$, is a  continuous
bilinear mapping;
\item[iv)] for all $\varphi\in\DD^{\{M_p\}}\left(\RR^d\right)$, $\left(\varphi*\check{T}\right)S\in\tilde{\DD}'^{\{M_p\}}_{L^1}\left(\RR^d\right)$ and for every compact subset $K$ of $\RR^d$,
$(\varphi,\chi)\mapsto \left\langle\left(\varphi*\check{T}\right)S,\chi\right\rangle$, $\DD^{\{M_p\}}_K\times\dot{\tilde{\mathcal{B}}}^{\{M_p\}}\left(\RR^d\right)\longrightarrow \CC$, is a continuous bilinear mapping;
\item[v)] for all $\varphi,\psi\in\DD^{\{M_p\}}\left(\RR^d\right)$, $\left(\varphi*\check{S}\right)(\psi*T)\in L^1\left(\RR^d\right)$.
\end{itemize}
\end{theorem}

\begin{proof}  $i)\Rightarrow ii)$.  Let $a>0$. Choose $\varphi\in\DD^{\{M_p\}}\left(\RR^d\right)$ such that $\varphi=1$ on
$K_{\RR^d}(0,a)$ and $\varphi=0$ on the complement of some bounded neighborhood of this set. Then, there exist $(t_j)\in\mathfrak{R}$ and $C>0$ such that $\left|\left\langle (S\otimes T)\varphi^{\Delta},\psi\right\rangle\right|\leq C\|\psi\|_{(t_j)}$ for all $\psi\in\DD^{\{M_p\}}_{\Delta_a}\left(\RR^{2d}\right)\subseteq \dot{\tilde{\mathcal{B}}}^{\{M_p\}}\left(\RR^{2d}\right)$. Since $\left\langle (S\otimes T)\varphi^{\Delta},\psi\right\rangle=\left\langle S\otimes T,\varphi^{\Delta}\psi\right\rangle
=\left\langle S\otimes T,\psi\right\rangle,$ it follows that $\left|\left\langle S\otimes T,\psi\right\rangle\right|\leq C\|\psi\|_{(t_j)}$ for all $\psi\in\DD^{\{M_p\}}_{\Delta_a}\left(\RR^{2d}\right)$. By lemma \ref{123}, it follows that $S\otimes T$ is a continuous linear mapping
from $\dot{\mathcal{B}}^{\{M_p\}}_a$ to $\CC$. Hence $S\otimes T\in \left(\dot{\mathcal{B}}^{\{M_p\}}_{\Delta}\right)'$.

$ii) \Rightarrow i)$. Let $\varphi\in\DD^{\{M_p\}}\left(\RR^d\right)$ with support in $K_{\RR^d}(0,a)$ for some $a>0$. Then, for that $a$, there exist $(t_j)\in\mathfrak{R}$ and $C>0$ such that $\left|\left\langle S\otimes T,\psi\right\rangle\right|\leq C\|\psi\|_{(t_j)}$ for all $\psi\in\dot{\mathcal{B}}^{\{M_p\}}_a$. Let $\psi\in\DD^{\{M_p\}}\left(\RR^{2d}\right)$. Then $\varphi^{\Delta}\psi\in\DD^{\{M_p\}}_{\Delta_a}\subseteq\dot{\mathcal{B}}^{\{M_p\}}_a$ and  by lemma \ref{7}
\begin{eqnarray*}
\left|\left\langle (S\otimes T)\varphi^{\Delta},\psi\right\rangle\right|=\left|\left\langle S\otimes T,\varphi^{\Delta}\psi\right\rangle\right|\leq C\left\|\varphi^{\Delta}\psi\right\|_{(t_j)}\leq \tilde{C}\left\|\psi\right\|_{(t'_j)},
\end{eqnarray*}
for some $(t'_j)\in\mathfrak{R}$ and $\tilde{C}>0$ that depend on $\varphi$ and $(t_j)$. Thus, $(S\otimes T)\varphi^{\Delta}\in\tilde{\DD}'^{\{M_p\}}_{L^1}$.\\
$ii)\Rightarrow iii)$.  Let $F$ and $K_1$ be compact subsets of $\RR^d$. Take $K$ to be a compact set in $\RR^d$ such that $F\subset\subset \mathrm{int}K$ and let $\varphi\in\DD^{\{M_p\}}_{K_1}$, $\psi\in\DD^{\{M_p\}}_K$ and $\chi\in\DD^{\{M_p\}}\left(\RR^d\right)$. Then
$$
\left\langle\left(\left(\varphi*\check{S}\right)T\right)*\psi,\chi\right\rangle=
\left\langle S\otimes T,\varphi(x+y)\left(\check{\psi}*\chi\right)(y)\right\rangle.
$$
There exists $a>0$ such that $\mathrm{supp\,} \varphi^{\Delta}(x,y)\left(\check{\psi}*\chi\right)(y)\subseteq\Delta_a$, for all $\varphi\in\DD^{\{M_p\}}_{K_1}$, $\psi\in\DD^{\{M_p\}}_K$ and $\chi\in\DD^{\{M_p\}}\left(\RR^d\right)$. Then, for that $a$, there exist $(t_j)\in\mathfrak{R}$ and $C_1>0$ such that $\left|\langle S\otimes T,\theta\rangle\right|\leq C_1\|\theta\|_{(t_j)}$ for all $\theta\in\dot{\mathcal{B}}^{\{M_p\}}_a$. So we obtain
$$
\left|\left\langle\left(\left(\varphi*\check{S}\right)T\right)*\psi,\chi\right\rangle\right|= C_1\left\|\varphi^{\Delta}(x,y)\left(\check{\psi}*\chi\right)(y)\right\|_{(t_j)}.
$$
We have
\begin{eqnarray*}
\frac{\left|D^{\alpha}_x D^{\beta}_y\left(\varphi^{\Delta}(x,y)\left(\check{\psi}*\chi\right)(y)\right)\right|} {T_{\alpha+\beta}M_{\alpha+\beta}}&\leq& \sum_{\delta\leq\beta}{\beta\choose\delta} \frac{\left|D^{\alpha+\beta-\delta}\varphi(x+y)\right|\left|D^{\delta}\left(\check{\psi}*\chi\right)(y)\right|} {T_{\alpha+\beta}M_{\alpha+\beta}}\\
&\leq& \|\varphi\|_{(t_j/2)}\sum_{\delta\leq\beta}{\beta\choose\delta} \frac{\left|D^{\delta}\left(\check{\psi}*\chi\right)(y)\right|}{2^{|\alpha|+|\beta|-|\delta|}T_{\delta}M_{\delta}}\\
&\leq& |K|\|\varphi\|_{(t_j/2)}\|\psi\|_{(t_j/2)}\|\chi\|_{L^{\infty}}.
\end{eqnarray*}
Hence $\left|\left\langle\left(\left(\varphi*\check{S}\right)T\right)*\psi,\chi\right\rangle\right|\leq
C_1|K|\|\varphi\|_{(t_j/2)}\|\psi\|_{(t_j/2)}\|\chi\|_{L^{\infty}}$, for all $\varphi\in\DD^{\{M_p\}}_{K_1}$, $\psi\in\DD^{\{M_p\}}_K$ and $\chi\in\DD^{\{M_p\}}\left(\RR^d\right)$. Thus,  $\left(\left(\varphi*\check{S}\right)T\right)*\psi\in\mathcal{M}^1$. Since $\left(\left(\varphi*\check{S}\right)T\right)*\psi\in \EE^{\{M_p\}}\left(\RR^d\right)$, it follows that $\left(\left(\varphi*\check{S}\right)T\right)*\psi\in L^1$. Let  $\varphi\in\DD^{\{M_p\}}_{K_1}$ be fixed. Then the mapping $\psi\mapsto\left(\left(\varphi*\check{S}\right)T\right)*\psi$, $\DD^{\{M_p\}}_K\longrightarrow \DD'^{\{M_p\}}$ is continuous and has a closed graph. Since $\left(\left(\varphi*\check{S}\right)T\right)*\psi\in L^1$ and if we consider the above  mapping from $\DD^{\{M_p\}}_K$ to $L^1$, it has a closed graph and so, it is continuous. ($L^1$ is a $(B)$ - space and $\DD^{\{M_p\}}_K$ is a $(DFS)$ - space.) Hence, there exist $(r_j)\in\mathfrak{R}$ and $C_1>0$ such that
\begin{eqnarray}\label{30}
\left\|\left(\left(\varphi*\check{S}\right)T\right)*\psi\right\|_{L^1}\leq C_1\|\psi\|_{K,(r_j)}.
\end{eqnarray}
By lemma \ref{3}, we can assume, without losing generality, that $(r_j)$ is such that $R_{j+k}\leq 2^{j+k}R_j R_k$, for all $j,k\in\NN$. Let $r'_j=r_j/(2H)$ and $\theta\in\DD^{M_p}_{F,(r'_j)}$. Then, there exist $\psi_n\in \DD^{\{M_p\}}_K$, $n\in\ZZ_+$ such that $\psi_n\longrightarrow\theta$ in $\DD^{M_p}_{K,(r_j)}$. The mapping $\theta\mapsto\left(\left(\varphi*\check{S}\right)T\right)*\theta$, $\DD^{M_p}_{K,(r_j)}\longrightarrow \DD'^{\{M_p\}}$ is continuous. So, if $\psi_n\in \DD^{\{M_p\}}_K$ tends to $\theta\in\DD^{M_p}_{F,(r'_j)}$ in the topology of $\DD^{M_p}_{K,(r_j)}$ then $\left(\left(\varphi*\check{S}\right)T\right)*\psi_n\longrightarrow \left(\left(\varphi*\check{S}\right)T\right)*\theta$ in $\DD'^{\{M_p\}}$. By (\ref{30}), we have $\left\|\left(\left(\varphi*\check{S}\right)T\right)*\psi_n\right\|_{L^1}\leq C_1\|\psi_n\|_{K,(r_j)}$. So, $\left(\left(\varphi*\check{S}\right)T\right)*\psi_n$ is a Cauchy sequence in $L^1$, hence it must be convergent and it must converge to $\left(\left(\varphi*\check{S}\right)T\right)*\theta$, because it converge to that ultradistribution in $\DD'^{\{M_p\}}$. Consequently, $\left(\left(\varphi*\check{S}\right)T\right)*\theta\in L^1$ for all $\theta\in \DD^{M_p}_{F,(r'_j)}$ and if we let $n\longrightarrow \infty$ in the last inequality, we get $\ds\left\|\left(\left(\varphi*\check{S}\right)T\right)*\theta\right\|_{L^1}\leq C_1\|\theta\|_{K,(r_j)}$, for all $\theta\in \DD^{M_p}_{F,(r'_j)}$. By corollary 1 of \cite{PilipovicT}, it follows that $\left(\varphi*\check{S}\right)T\in\tilde{\DD}'^{\{M_p\}}_{L^1}$. Now, we  prove that the mapping $(\varphi,\chi)\mapsto \left\langle\left(\varphi*\check{S}\right)T,\chi\right\rangle$, $\DD^{\{M_p\}}_K\left(\RR^d\right)\times \dot{\tilde{\mathcal{B}}}^{\{M_p\}}\left(\RR^d\right) \longrightarrow \CC$, is continuous, for every compact set $K$. There exists $a>0$ such that $K\subset\subset _{\RR^d}(0,a)$. Take $\theta\in\DD^{\{M_p\}}$ such that $\theta=1$ on $K_{\RR^d}(0,a)$ and $\theta=0$ on the complement of some bounded neighborhood of this ball. Then $\varphi^{\Delta}\theta^{\Delta}=\varphi^{\Delta}$ for all $\varphi\in\DD^{\{M_p\}}_K$. Let $\varphi\in \DD^{\{M_p\}}_K$ and $\chi,\psi_n\in\DD^{\{M_p\}}$, $n\in\ZZ_+$, such that $\psi_n\longrightarrow \delta$, when $n$ tends to infinity, in $\EE'^{\{M_p\}}$.  Then
\begin{eqnarray*}
\left\langle \left(\varphi*\check{S}\right)T,\chi\right\rangle&=&\lim_{n\rightarrow\infty}\left\langle \left(\left(\varphi*\check{S}\right)T\right)*\psi_n,\chi\right\rangle\\
&=&\lim_{n\rightarrow\infty}\left\langle S\otimes T,\varphi^{\Delta}(x,y)\left(\check{\psi}_n*\chi\right)(y)\right\rangle
=\left\langle S\otimes T,\varphi^{\Delta}(x,y)\chi(y)\right\rangle\\
&=&\left\langle S\otimes T,\varphi^{\Delta}(x,y)\theta^{\Delta}(x,y)\chi(y)\right\rangle
=\left\langle (S\otimes T)\varphi^{\Delta},1_x\otimes\chi(y)\right\rangle,
\end{eqnarray*}
where the last tow terms are in the sense of the duality $\left\langle\tilde{\tilde{\DD}}^{\{M_p\}}_{L^{\infty}}, \left(\tilde{\tilde{\DD}}^{\{M_p\}}_{L^{\infty}}\right)'\right\rangle$. Now, let $\chi\in\dot{\tilde{\mathcal{B}}}^{\{M_p\}}\left(\RR^d\right)$ and take $\psi\in\DD^{\{M_p\}}\left(\RR^d\right)$ such that $\psi=1$ on $K_{\RR^d}(0,1)$ and $\psi=0$ out of $K_{\RR^d}(0,2)$. Put $\psi_n(x)=\psi(x/n)$, $n\in\ZZ_+$, and $\chi_n(x)=\psi_n(x)\chi(x)$. Then, one easily checks that $1_x\otimes\chi_n(y)\longrightarrow 1_x\otimes \chi(y)$ in $\tilde{\tilde{\DD}}^{\{M_p\}}_{L^{\infty}}\left(\RR^{2d}\right), n\longrightarrow\infty$. Because $\left(\varphi*\check{S}\right)T\in\left(\tilde{\tilde{\DD}}^{\{M_p\}}_{L^{\infty}}\right)'\left(\RR^d\right) =\tilde{\DD}'^{\{M_p\}}_{L^1}\left(\RR^d\right)$ and $(S\otimes T) \varphi^{\Delta}\in\left(\tilde{\tilde{\DD}}^{\{M_p\}}_{L^{\infty}}\right)'\left(\RR^{2d}\right) =\tilde{\DD}'^{\{M_p\}}_{L^1}\left(\RR^{2d}\right)$ (c.f. proposition \ref{25}), we have
\begin{eqnarray*}
\left\langle \left(\varphi*\check{S}\right)T,\chi\right\rangle&=&\lim_{n\rightarrow\infty}\left\langle \left(\varphi*\check{S}\right)T,\chi_n\right\rangle=\lim_{n\rightarrow\infty}\left\langle (S\otimes T)\varphi^{\Delta},1_x\otimes\chi_n(y)\right\rangle\\
&=&\left\langle (S\otimes T)\varphi^{\Delta},1_x\otimes\chi(y)\right\rangle, \varphi\in \DD^{\{M_p\}}_K, \chi\in\dot{\tilde{\mathcal{B}}}^{\{M_p\}}\left(\RR^d\right).
\end{eqnarray*}
 Also $(S\otimes T)\theta^{\Delta}\in\left(\tilde{\tilde{\DD}}^{\{M_p\}}_{L^{\infty}}\right)'\left(\RR^{2d}\right)$ and by the construction of $\theta$, $(S\otimes T)\theta^{\Delta}\varphi^{\Delta}=(S\otimes T)\varphi^{\Delta}$. Hence
\begin{eqnarray}\label{40}
\left\langle \left(\varphi*\check{S}\right)T,\chi\right\rangle=\left\langle (S\otimes T)\theta^{\Delta},\varphi^{\Delta}(x,y)
\chi(y)\right\rangle, \varphi\in \DD^{\{M_p\}}_K, \chi\in\dot{\tilde{\mathcal{B}}}^{\{M_p\}}\left(\RR^d\right).
\end{eqnarray}
Since the bilinear mapping $$(\varphi(x),\chi(y))\mapsto\varphi^{\Delta}(x,y)\chi(y), \DD^{\{M_p\}}_K\times \dot{\tilde{\mathcal{B}}}^{\{M_p\}}\left(\RR^d\right)\longrightarrow \tilde{\tilde{\DD}}^{\{M_p\}}_{L^{\infty}}\left(\RR^{2d}\right)$$ is continuous and  $(S\otimes T)\theta^{\Delta}\in\left(\tilde{\tilde{\DD}}^{\{M_p\}}_{L^{\infty}}\right)'\left(\RR^{2d}\right)$, it follows that the bilinear mapping $$(\varphi(x),\chi(y))\mapsto\left\langle (S\otimes T)\theta^{\Delta},\varphi^{\Delta}(x,y)\chi(y)\right\rangle, \DD^{\{M_p\}}_K\times \dot{\tilde{\mathcal{B}}}^{\{M_p\}}\left(\RR^d\right)\longrightarrow \CC$$ is continuous. Hence, by (\ref{40}), we obtain the desired continuity.\\
\indent $ii)\Rightarrow iv)$ The proof is analogous to $ii)\Rightarrow iii)$.\\
\indent $ii)\Rightarrow v)$. Let  $K\subset\subset\RR^d$ and let $\varphi,\psi\in\DD^{\{M_p\}}_K$, $\chi\in\DD^{\{M_p\}}$. Then
\begin{eqnarray*}
\left\langle \left(\varphi*\check{S}\right)(\psi*T),\chi\right\rangle&=&\left\langle\langle S(x),\varphi(x+t)\rangle\langle T(y),\psi(t-y)\rangle,\chi(t)\right\rangle\\
&=&\left\langle \left((S\otimes T)(x,y)\right)\otimes 1_t,\varphi(x+t)\psi(t-y)\chi(t)\right\rangle\\
&=&\left\langle (S\otimes T)(x,y),\int_{\RR^d}\varphi(x+t)\psi(t-y)\chi(t)dt\right\rangle.
\end{eqnarray*}
Let $\ds\theta(x,y)=\int_{\RR^d}\varphi(x+t)\psi(t-y)\chi(t)dt$. Let $a>0$ be such that $K\subset\subset
K_{\RR^d}(0,a)$. One prove that $\mathrm{supp\,} \theta\subseteq \Delta_{2a}$.  Now, because $\varphi,\psi\in\DD^{\{M_p\}}_K$, there exist $h_1,h_2,C_1,C_2>0$ such that $\left|D^{\alpha}\varphi(x)\right|\leq C_1 h^{|\alpha|}_1M_{\alpha}$ and $\left|D^{\alpha}\psi(x)\right|\leq C_2 h^{|\alpha|}_2M_{\alpha}$. Let $(t_j)\in\mathfrak{R}$. We have
\begin{eqnarray*}
\frac{\left|D^{\alpha}_x D^{\beta}_y \theta(x,y)\right|}{T_{\alpha+\beta}M_{\alpha+\beta}}&\leq& \int_{\RR^d}\frac{\left|D^{\alpha}\varphi(x+t)\right|\left|D^{\beta}\psi(t-y)\right||\chi(t)|} {T_{\alpha+\beta}M_{\alpha+\beta}}dt\\
&\leq&\|\chi\|_{L^{\infty}}\int_{K} \frac{\left|D^{\alpha}\varphi(x+y+t)\right|\left|D^{\beta}\psi(t)\right|}{T_{\alpha+\beta}M_{\alpha+\beta}}dt\leq C_1C_2C_3|K|\|\chi\|_{L^{\infty}}.
\end{eqnarray*}
It follows that the mapping $\ds\chi\mapsto\int_{\RR^d}\varphi(x+t)\psi(t-y)\chi(t)dt$, $\mathcal{C}_0\left(\RR^d\right)\longrightarrow \dot{\mathcal{B}}^{\{M_p\}}_{2a}$ is continuous, i.e. this mapping is continuous as a mapping from $\mathcal{C}_0\left(\RR^d\right)$ to $\dot{\mathcal{B}}^{\{M_p\}}_{\Delta}$. But, $S\otimes T\in\left(\dot{\mathcal{B}}^{\{M_p\}}_{\Delta}\right)'$, so the mapping
$$\ds\chi\mapsto \left\langle (S\otimes T)(x,y),\int_{\RR^d}\varphi(x+t)\psi(t-y)\chi(t)dt\right\rangle, \;  \mathcal{C}_0\left(\RR^d\right)\longrightarrow \CC,$$ is continuous. Since $\left(\varphi*\check{S}\right)(\psi*T)\in \mathcal{M}^1$ and it belongs to $\EE^{\{M_p\}}$, it follows $\left(\varphi*\check{S}\right)(\psi*T)\in L^1$.\\
\indent  $iii)\Rightarrow i)$.  Let $\varphi\in\DD^{\{M_p\}}\left(\RR^d\right)$  and let $K\subset\subset\RR^d$ such that $\mathrm{supp\,} \varphi\subset\subset \mathrm{int}K$. By the assumption, the bilinear mapping $G:\DD^{\{M_p\}}_K\times \dot{\tilde{\mathcal{B}}}^{\{M_p\}}\left(\RR^d\right)\longrightarrow \CC$, $G(\psi,\chi)= \left\langle \left((\psi\varphi)*\check{S}\right)T,\chi\right\rangle$, is continuous. Hence $G$ extends to a linear continuous mapping, $\hat{G}$, on the completion of the tensor product $\DD^{\{M_p\}}_K\hat{\otimes} \dot{\tilde{\mathcal{B}}}^{\{M_p\}}\left(\RR^d\right)$ ($\DD^{\{M_p\}}_K$ is nuclear and the $\pi$ topology coincides with the $\epsilon$ topology). Let $\theta\in\DD^{\{M_p\}}_K$ be a function such that $\theta=1$ on $\mathrm{supp\,} \varphi$. Then, the mapping $F:\dot{\tilde{\mathcal{B}}}^{\{M_p\}}\left(\RR^d\right)\longrightarrow \DD^{\{M_p\}}_K$, $F(\chi)=\theta\chi$ is continuous. So, the mapping
\begin{eqnarray*}
F\otimes_{\epsilon} \mathrm{Id}:\dot{\tilde{\mathcal{B}}}^{\{M_p\}}\left(\RR^d\right)\otimes_{\epsilon} \dot{\tilde{\mathcal{B}}}^{\{M_p\}}\left(\RR^d\right)\longrightarrow \DD^{\{M_p\}}_K\otimes_{\epsilon}\dot{\tilde{\mathcal{B}}}^{\{M_p\}}\left(\RR^d\right)
\end{eqnarray*}
is continuous and by proposition \ref{100}, we have the continuous extension  $F\hat{\otimes}_{\epsilon} \mathrm{Id}:\dot{\tilde{\mathcal{B}}}^{\{M_p\}}\left(\RR^{2d}\right)\longrightarrow \DD^{\{M_p\}}_K\hat{\otimes}_{\epsilon}\dot{\tilde{\mathcal{B}}}^{\{M_p\}}\left(\RR^d\right)$. Thus, we have the continuous mapping $$\ds\tilde{G}:\dot{\tilde{\mathcal{B}}}^{\{M_p\}}\left(\RR^{2d}\right)\xrightarrow{F\hat{\otimes}_{\epsilon} Id} \DD^{\{M_p\}}_K \hat{\otimes}_{\epsilon} \dot{\tilde{\mathcal{B}}}^{\{M_p\}}\left(\RR^d\right)\xrightarrow{\hat{G}} \CC,$$ i.e. $\tilde{G}\in\tilde{\DD}'^{\{M_p\}}_{L^1}\left(\RR^{2d}\right)$. For $\psi,\chi\in\DD^{\{M_p\}}\left(\RR^d\right)$,
\begin{eqnarray*}
\tilde{G}(\psi\otimes\chi)&=&\hat{G}\left(F(\psi)\otimes\chi\right)=G(\theta\psi,\chi)= \left\langle\left((\theta\psi\varphi)*\check{S}\right)T,\chi\right\rangle\\
&=&\langle (S\otimes T) \varphi^{\Delta},\psi(x+y)\chi(y)\rangle.
\end{eqnarray*}
Let $\Theta$ be the linear transformation $\Theta(x,y)=(x+y,y)$ and denote by $\tilde{\Theta}$ the linear operator $\tilde{\Theta}f(x',y')=f\circ\Theta(x,y)=f(x+y,y)$. It is obviously an isomorphism of $\DD^{\{M_p\}}\left(\RR^{2d}\right)$ and of $\dot{\tilde{\mathcal{B}}}^{\{M_p\}}\left(\RR^{2d}\right)$, hence, the transposed mapping $^t\tilde{\Theta}$ is an isomorphism of $\DD'^{\{M_p\}}\left(\RR^{2d}\right)$ and of $\tilde{\DD}'^{\{M_p\}}_{L^1}\left(\RR^{2d}\right)$. It follows that $\tilde{G}(\psi\otimes\chi)=\left\langle (S\otimes T)\varphi^{\Delta},\tilde{\Theta}(\psi\otimes\chi)\right\rangle=\left\langle {}^t\tilde{\Theta}\left((S\otimes T)\varphi^{\Delta}\right),\psi\otimes\chi\right\rangle$. Because $\DD^{\{M_p\}}\left(\RR^d\right)\otimes\DD^{\{M_p\}}\left(\RR^d\right)$ is dense in $\DD^{\{M_p\}}\left(\RR^{2d}\right)$, $\tilde{G}={}^t\tilde{\Theta}\left((S\otimes T) \varphi^{\Delta}\right)$ in $\DD'^{\{M_p\}}\left(\RR^{2d}\right)$. $\tilde{G}\in \tilde{\DD}'^{\{M_p\}}_{L^1}\left(\RR^{2d}\right)$, so ${}^t\tilde{\Theta}\left((S\otimes T)\varphi^{\Delta}\right)\in\tilde{\DD}'^{\{M_p\}}_{L^1}\left(\RR^{2d}\right)$, i.e. $(S\otimes T)\varphi^{\Delta}\in \tilde{\DD}'^{\{M_p\}}_{L^1}\left(\RR^{2d}\right)$.\\
\indent $iv)\Rightarrow i)$ The proof  is analogous to the previous one.\\
\indent $v)\Rightarrow i)$. Let $K$ and $K_1$ be compact subsets of $\RR^d$ such that $K_1\subset\subset \mathrm{int}K$ and both satisfy the cone property (for the definition of the cone property see \cite{Komatsu2}, p. 614). Observe the mapping $G:\DD^{\{M_p\}}_K\times \DD^{\{M_p\}}_K\longrightarrow \mathcal{M}^1$, $G(\varphi,\psi)=\left(\varphi*\check{S}\right)(\psi*T)$. Note that the mapping $\varphi\mapsto\left(\varphi*\check{S}\right)(\psi*T)$ is continuous from $\DD^{\{M_p\}}_K$ to $\DD'^{\{M_p\}}$ and hence, it has a closed graph. Because $\mathcal{M}^1$ is a $(B)$ - space and $\DD^{\{M_p\}}_K$ is a $(DFS)$ - space, from the closed graph theorem, it follows that $G$ is separately continuous in $\varphi$ and similarly in $\psi$. $\DD^{\{M_p\}}_K$ is a $(DFS)$ - space, hence $G$ is continuous. It can be extended to a continuous mapping, $\hat{G}$, on the completion of the tensor product $\DD^{\{M_p\}}_K\hat{\otimes} \DD^{\{M_p\}}_K$. Since $\DD^{\{M_p\}}_K\hat{\otimes} \DD^{\{M_p\}}_K\cong\DD^{\{M_p\}}_{K\times K}$ (theorem 2.1. of \cite{Komatsu2}), the mapping $\DD^{\{M_p\}}_{K\times K}\times \mathcal{C}_0\left(\RR^d\right)\longrightarrow \CC$, $(f,\theta)\mapsto \langle \hat{G}(f),\theta\rangle$, is continuous because it is the composition of the mappings
\begin{eqnarray*}
\DD^{\{M_p\}}_{K\times K}\times \mathcal{C}_0\left(\RR^d\right)\xrightarrow{\hat{G}\times Id} \mathcal{M}^1\left(\RR^d\right)\times \mathcal{C}_0\left(\RR^d\right)\xrightarrow{\langle \cdot, \cdot\rangle} \CC,
\end{eqnarray*}
where the last mapping is the duality of $\mathcal{C}_0$ and $\mathcal{M}^1$. Hence,  the mapping $\DD^{\{M_p\}}_{K\times K}\times \dot{\tilde{\mathcal{B}}}^{\{M_p\}}\left(\RR^d\right)\longrightarrow \CC$, $(f,\chi)\mapsto \langle \hat{G}(f),\chi\rangle$, is continuous. So, this mapping can be extended to $\tilde{G}$ on the completion of the tensor product $\DD^{\{M_p\}}_{K\times K}\hat{\otimes} \dot{\tilde{\mathcal{B}}}^{\{M_p\}}$.  Take $\theta\in\DD^{\{M_p\}}_K$ such that $\theta=1$ on $K_1$ and put $\theta_1(x)=\theta(x)$ and $\theta_2(y)=\theta(y)$.  Because  $\psi\mapsto \theta_1\theta_2\psi$, $\dot{\tilde{\mathcal{B}}}^{\{M_p\}}\left(\RR^{2d}\right)\longrightarrow \DD^{\{M_p\}}_{K\times K}$, is continuous, the mapping $\psi\otimes \varphi\mapsto \theta_1\theta_2\psi\otimes \varphi$, $\dot{\tilde{\mathcal{B}}}^{\{M_p\}}\left(\RR^{2d}\right)\otimes_{\epsilon} \dot{\tilde{\mathcal{B}}}^{\{M_p\}}\left(\RR^d\right) \longrightarrow\DD^{\{M_p\}}_{K\times K}\otimes_{\epsilon} \dot{\tilde{\mathcal{B}}}^{\{M_p\}}\left(\RR^d\right)$ is continuous and it extends to a continuous mapping $V$ on the completion
of these spaces. By proposition \ref{100}, the composition $\tilde{G}\circ V$ is  continuous from $\dot{\tilde{\mathcal{B}}}^{\{M_p\}}\left(\RR^{3d}\right)$ to $\CC$. That means that there exist $(t_j)\in\mathfrak{R}$ and $C_1>0$ such that $\left|\tilde{G}\circ V(f)\right|\leq C_1\|f\|_{(t_j)}$, for all $f\in \dot{\tilde{\mathcal{B}}}^{\{M_p\}}\left(\RR^{3d}\right)$. Let $\varphi,\psi,\chi\in \DD^{\{M_p\}}\left(\RR^d\right)$, then
\begin{eqnarray*}
\tilde{G}\circ V(\varphi\otimes \psi\otimes \chi)&=& \tilde{G}(\theta_1\varphi\otimes \theta_2 \psi\otimes \chi)
=\left\langle\left((\theta_1\varphi)*\check{S}\right)((\theta_2\psi)*T),\chi\right\rangle\\
&=&\langle (S(x)\otimes T(y))\otimes 1_t,\theta_1(x+t)\varphi(x+t)\theta_2(t-y)\psi(t-y)\chi(t)\rangle.
\end{eqnarray*}
By nuclearity and theorem 2.1. of \cite{Komatsu2}, we have continuous dense inclusions
$$
\left(\DD^{\{M_p\}}\left(\RR^d\right)\otimes \DD^{\{M_p\}}\left(\RR^d\right)\right)\otimes \DD^{\{M_p\}}\left(\RR^d\right)\longrightarrow \DD^{\{M_p\}}\left(\RR^{3d}\right).$$
So, for $\tilde{\varphi}\in \DD^{\{M_p\}}\left(\RR^{3d}\right)$, there exists a net $\tilde{\varphi}_{\nu}\in \left(\DD^{\{M_p\}}\left(\RR^d\right)\otimes \DD^{\{M_p\}}\left(\RR^d\right)\right)\otimes \DD^{\{M_p\}}\left(\RR^d\right)$ such that $\tilde{\varphi}_{\nu}\longrightarrow \tilde{\varphi}$ in $\DD^{\{M_p\}}\left(\RR^{3d}\right)$. But then the convergence holds in $\dot{\tilde{\mathcal{B}}}^{\{M_p\}}\left(\RR^{3d}\right)$ and, for $\tilde{\varphi}\in \DD^{\{M_p\}}\left(\RR^{3d}\right),$
\begin{eqnarray*}
\tilde{G}\circ V(\tilde{\varphi})=\langle (S(x)\otimes T(y))\otimes 1_t,\theta_1(x+t)\theta_2(t-y)\tilde{\varphi}(x+t,t-y,t)\rangle.
\end{eqnarray*}
 Let $\varphi\in\DD^{\{M_p\}}\left(\RR^d\right)$, $K_1=K_{\RR^d}(0,a)$, where $ a>0$ is such that $\mathrm{supp\,} \varphi\subset\subset \mathrm{int}K_1$. Let $K=K_{\RR^d}(0,a+2)$ and $K'=K_{\RR^d}(0,a+1)$. Choose $\theta\in\DD^{\{M_p\}}\left(\RR^d\right)$ to be equal to $1$ on $K'$ and has a support in $\mathrm{int}K$.  Take $\mu\in\DD^{\{M_p\}}\left(\RR^d\right)$ with support in the open unit ball and $\ds \int_{\RR^d}\mu(x)dx=1$. Let $\chi\in\DD^{\{M_p\}}\left(\RR^{2d}\right)$ be arbitrary and consider the function $f(x,y,t)=\varphi(x-y)\chi(x-t,t-y)\mu(x)$. Obviously $f\in \DD^{\{M_p\}}\left(\RR^{3d}\right)$ and
\begin{eqnarray*}
\tilde{G}\circ V(f)&=&\langle (S(x)\otimes T(y))\otimes 1_t,\theta_1(x+t)\theta_2(t-y)f(x+t,t-y,t)\rangle\\
&=&\langle (S(x)\otimes T(y))\otimes 1_t,\theta_1(x+t)\theta_2(t-y)\varphi(x+y)\chi(x,y)\mu(x+t)\rangle.
\end{eqnarray*}
By construction $\theta_1(x+t)\mu(x+t)=\mu(x+t)$, for all $x,t\in\RR^d$. Let $x,y,t\in\RR^d$ are such that $\varphi(x+y)\mu(x+t)\neq 0$. Then $|x+y|<a$ and $|x+t|<1$. So, $|t-y|\leq |x+y|+|x+t|<a+1$, hence $\theta_2(t-y)=1$. We have
\begin{eqnarray*}
\tilde{G}\circ V(f)&=&\langle (S(x)\otimes T(y))\otimes 1_t,\varphi(x+y)\chi(x,y)\mu(x+t)\rangle\\
&=&\left\langle S(x)\otimes T(y),\varphi(x+y)\chi(x,y)\int_{\RR^d}\mu(x+t)dt\right\rangle\\
&=&\left\langle (S(x)\otimes T(y))\varphi^{\Delta}(x,y), \chi(x,y)\right\rangle.
\end{eqnarray*}
One easily obtains the following estimate for the derivatives of $f$:
\begin{eqnarray*}
\frac{\left|D^{\alpha}_x D^{\beta}_y D^{\gamma}_t f(x,y,t)\right|}{T_{\alpha+\beta+\gamma}M_{\alpha+\beta+\gamma}}\leq \|\varphi\|_{(t_j/4)}\|\mu\|_{(t_j/4)}\|\chi\|_{(t_j/4)}.
\end{eqnarray*}
Hence,
$$
\left|\left\langle (S(x)\otimes T(y))\varphi^{\Delta}(x,y), \chi(x,y)\right\rangle\right|=\left|\tilde{G}\circ V(f)\right|\leq C_1\|f\|_{(t_j)}$$
$$
\leq C_1\|\varphi\|_{(t_j/4)}\|\mu\|_{(t_j/4)}\|\chi\|_{(t_j/4)}, \chi\in\DD^{\{M_p\}}\left(\RR^{2d}\right).
$$
 Since $\DD^{\{M_p\}}\left(\RR^{2d}\right)$ is dense in $\dot{\tilde{\mathcal{B}}}^{\{M_p\}}\left(\RR^{2d}\right)$, the proof follows.\qed
\end{proof}

\begin{remark}
Let $\chi\in\DD^{\{M_p\}}\left(\RR^d\right)$ is equal to $1$ on the $K_{\RR^d}(0,1)$ and has a support in $K_{\RR^d}(0,2)$. Put $\chi_n(x)=\chi(x/n)$, $n\in\ZZ_+$. If for $S$ and $T$ the equivalent conditions of the above theorem hold and $\varphi\in\DD^{\{M_p\}}\left(\RR^d\right)$, then, similarly as in the proof of $ii)\Rightarrow iii)$, we can prove that $\left\langle \left(\varphi*\check{S}\right)T,\chi_n\right\rangle=\left\langle (S\otimes T) \varphi^{\Delta},1_x\otimes\chi_n(y)\right\rangle$. But then, by construction, $\chi_n\longrightarrow 1$ in $\tilde{\tilde{\DD}}^{\{M_p\}}_{L^{\infty}}\left(\RR^d\right)$ and $1_x\otimes \chi_n(y)\longrightarrow 1_{x,y}$ in $\tilde{\tilde{\DD}}^{\{M_p\}}_{L^{\infty}}\left(\RR^{2d}\right)$. Hence $\langle S*T,\varphi\rangle=\left\langle (S\otimes T) \varphi^{\Delta},1\right\rangle=\left\langle \left(\varphi*\check{S}\right)T,1\right\rangle$. Similarly, $\langle S*T,\varphi\rangle=\left\langle \left(\varphi*\check{T}\right)S,1\right\rangle$.
\end{remark}


\end{document}